\documentclass[12pt, oneside, psamsfonts]{amsart}

\newif\ifPDF
\ifx\pdfoutput\undefined\PDFfalse
\else \ifnum \pdfoutput > 0 \PDFtrue
        \else \PDFfalse
        \fi
\fi

\usepackage[centertags]{amsmath}
\usepackage{amsfonts}
\usepackage{mathrsfs}
\usepackage{textcomp}
\usepackage{amssymb}
\usepackage{amsthm}
\usepackage{newlfont}
\usepackage[all]{xy}

\ifPDF
  \usepackage[pdftex]{color, graphicx}
  \usepackage[pdftex, bookmarks, colorlinks]{hyperref}
  \hypersetup{colorlinks=false}

\else
  \usepackage{color}
  \usepackage[dvips]{graphicx}
  \usepackage[dvips]{hyperref}
\fi

\usepackage[scale=0.8]{geometry}

\newtheorem{thm}{Theorem}[section]
\newtheorem{cor}[thm]{Corollary}
\newtheorem{lem}[thm]{Lemma}

\theoremstyle{definition}
\newtheorem{defn}[thm]{Definition}
\theoremstyle{remark}
\newtheorem{rem}[thm]{Remark}

\numberwithin{equation}{section}

\newcommand{\norm}[1]{\Vert#1 \Vert}
\newcommand{\abs}[1]{| #1 |}

\newcommand{\Real}{\mathbb R}
\newcommand{\Int}{\mathbb Z}
\newcommand{\Comp}{\mathbb C}
\newcommand{\Ratn}{\mathbb Q}
\newcommand{\eps}{\varepsilon}

\newcommand{\Kzero}{\mathrm{K}_0}
\newcommand{\Kone}{\mathrm{K}_1}

\begin{document}

\title{Remarks on Villadsen algebras}

\author{George A. Elliott}
\address{Department of Mathematics, University of Toronto, Toronto, Ontario, Canada~\ M5S 2E4}
\email{elliott@math.toronto.edu}

\author{Chun Guang Li}
\address{College of Mathematics and Statistics, Northeast Normal University, Changchun, Jilin, China, 130024}
\email{licg864@nenu.edu.cn}

\author{Zhuang Niu}
\address{Department of Mathematics and Statistics, University of Wyoming, Laramie, Wyoming 82071, USA}
\email{zniu@uwyo.edu}




\begin{abstract}
It is shown that certain unital simple C*-algebras constructed by Villadsen in \cite{Vill-perf} are classified by the $\Kzero$-group together with the radius of comparison.
\end{abstract}

\maketitle

\section{Introduction}
Villadsen algebras (of the first type) were constructed in \cite{Vill-perf} as examples of simple unital C*-algebras which have perforation in their ordered $\Kzero$-group. This class of C*-algebras lies outside the scope of the current classification theorem (\cite{GLN-TAS-1}, \cite{GLN-TAS-2}, \cite{EN-K0-Z}, \cite{EGLN-DR}, \cite{EGLN-ASH}, \cite{TWW-QD}, \cite{CETWW-dim-n}), as Villadsen algebras do not absorb the Jiang-Su algebra $\mathcal Z$ tensorially. Indeed, a Villadsen type algebra was constructed in \cite{Toms-Ann} which has the same value of the Elliott invariant as an AI algebra, but is not itself isomorphic to that AI algebra.

Each Villadsen algebra is an inductive limit of homogeneous C*-algebras with connecting maps induced by coordinate projections together with a small portion of point evaluations (see Section \ref{Vill-construction}). In this note, we shall first show that, with different point-evaluation sets, the resulting algebras are classified by the $\Kzero$-multiplicity and the radius of comparison of Toms (\cite{RC-Toms}; Definition \ref{RC} below):
\begin{thm}[Theorem \ref{thm-p}]
Let $X$ be a connected finite-dimensional solid space (see Definition \ref{solid-sp}). Let $A_E$ and $A_F$ be  Villadsen algebras (see Section \ref{Vill-construction}) with point-evaluation sets $E$ and $F$ respectively (but with the same connected space $X$ and the same numbers and multiplicities of the coordinate projections $(c_i)$, $(s_{i_1}, ..., s_{i, c_i})$). 
Then $A_E \cong A_F$ if, and only if, $$ \rho(\Kzero(A_E)) = \rho(\Kzero(A_F)) \quad \textrm{and}\quad \mathrm{rc}(A_E) = \mathrm{rc}(A_F),$$
where $\rho_{A_E}$ and $\rho_{A_F}$ are the unique states of the order-unit groups $\Kzero(A_E)$ and $\Kzero(A_F)$, respectively, and $\mathrm{rc}(\cdot)$ denotes the radius of comparison.
\end{thm}

Moreover, if the fixed seed space $X$ is further assumed to be K-contractible (that is, $\Kzero(\mathrm{C}(X)) = \Int$ and $\Kone(\mathrm{C}(X)) = \{0\}$), then the algebras can be classified by the $\Kzero$-group together with the radius of comparison even if the numbers and the multiplicities of the coordinate projections and the numbers of point evaluations are arbitrary:
\begin{thm}[Corollary \ref{diff-prod}]
Let $X$ be a connected finite-dimensional solid space which is K-contractible. Let 
$$A:=A(X^{p}, (n^{(A)}_i), (k^{(A)}_i), E^{(A)}) \quad\textrm{and}\quad B:=B(X^q, (n^{(B)}_i), (k^{(B)}_i), F^{(B)})$$
be Villadsen algebras with non-zero radius of comparison, where $p, q=1, 2, ...$.
Then 
$$A \cong B$$ 
if, and only if, $$\Kzero(A) \cong \Kzero(B) \quad \textrm{and}\quad \mathrm{rc}(A) = \mathrm{rc}(B).$$
\end{thm}
Note that this theorem covers the example constructed in \cite{Toms-Ann}.

One might compare the Villadsen algebras with the UHF algebras of  \cite{Glimm-UHF} and \cite{Dix}, and the present classification results with the classification of the unital UHF algebras, or, for that matter, of their non-unital hereditary subalgebras in \cite{Dix}. The non-unital version of the Villadsen algebras and their classification will  be considered in a forthcoming paper.

We hope that our result might shed some light on the possibility of classifying more general non-$\mathcal Z$-stable C*-algebras, for instance, general simple A(S)H algebras with diagonal maps (\cite{EHT-sr1} and \cite{Lutley-Alboiu}), or general simple  transformation group C*-algebras.

\subsection*{Acknowledgements} The research of the first named author was supported by a Natural Sciences
and Engineering Research Council of Canada (NSERC) Discovery Grant, the research of the second named author was supported by a National Natural Science Foundation of China (NNSF) grant (No.~11401088), and the research of the third named author was supported by a U.S.~National Science Foundation grant (DMS-1800882). The third named author also thanks Ali Asadi-Vasfi, Xuanlong Fu, and Cristian Ivanescu for discussions.

\section{The Villadsen algebra $A(X, (n_i), (k_i), E)$}\label{Vill-construction}
Let $X$ be a connected metrizable compact space, let $(c_i)$ and $(k_i)$ be two sequences of non-zero natural numbers, and let
$$
\left\{ 
\begin{array}{l}
E_1:=\{x_{1, 1}, ..., x_{1, k_1}\} \subseteq X,\\
E_2:=\{x_{2, 1}, ..., x_{2, k_2}\} \subseteq X^{c_1}, \\
 ... \\
 E_i: = \{x_{i, 1}, ..., x_{i, k_i}\} \subseteq X^{c_1 \cdots c_{i-1}}, \\
 ...
 \end{array}
\right.
$$
be a sequence of finite subsets such that for each $i=1, 2, ...$, the set
$$ \bigcup_{j=1}^\infty \bigcup_{s=1}^{c_{i+1} \cdots c_{i+j-1}}\pi_s (E_{i+j})$$
is dense in $X^{c_1\cdots c_i}$, where $\pi_s$ are the coordinate projections.

Construct the (generalized) Villadsen algebra as the inductive limit of the sequence 
\begin{equation}
\xymatrix{
\mathrm{M}_{n_0}(\mathrm{C}(X)) \ar[r] & \mathrm{M}_{n_0(n_1+k_1)}(\mathrm{C}(X^{c_1})) \ar[r] & \mathrm{M}_{n_0(n_1+k_1)(n_2+k_2)}(\mathrm{C}(X^{c_1c_2})) \ar[r] & \cdots ,
}
\end{equation}
where the seed for the $i$th-stage map,
$$\phi_i: \mathrm{C}(X^{c_1 \cdots c_{i-1}}) \to \mathrm{M}_{n_i+k_i}(\mathrm{C}(X^{c_1 \cdots c_{i-1}c_{i}})),$$ is defined by
\begin{eqnarray*} 
f & \mapsto & \mathrm{diag}\{ \underbrace{\underbrace{f\circ \pi_1, ..., f\circ \pi_1}_{s_{i, 1}}, ... ,  \underbrace{f\circ \pi_{c_i}, ..., f\circ\pi_{c_i}}_{s_{i, c_i}} }_{n_i}, \underbrace{f(x_{i, 1}), ..., f(x_{i, k_i})}_{k_i}\} \\
&& = \mathrm{diag}\{ \underbrace{\underbrace{f\circ \pi_1, ..., f\circ \pi_1}_{s_{i, 1}}, ... ,  \underbrace{f\circ \pi_{c_i}, ..., f\circ\pi_{c_i}}_{s_{i, c_i}} }_{n_i}, f(E_i) \},
\end{eqnarray*}
where $s_{i, 1}, ..., s_{i, c_i}\geq 1$ are natural numbers, and $n_i=\sum_{j=1}^{c_i} s_{i, j} $.

A direct calculation shows that the composed map $$\phi_{i, i+j}: \mathrm{C}(X^{c_1 \cdots c_{i-1}}) \to \mathrm{M}_{(n_i+k_i)\cdots(n_{i+j-1} + k_{i+j-1})}(\mathrm{C}(X^{c_1 \cdots c_{i-1}\cdots c_{i+j-1}})) $$ is equal (up to a permutation) to
$$f \mapsto \mathrm{diag}\{\underbrace{f\circ \pi_1, ..., f\circ\pi_{c_i \cdots c_{i+j-1}}}_{n_i \cdots n_{i+j-1}}, \underbrace{f(x_{i, 1}), ..., f(x_{i, k_i}),}_{k_i[(n_{i+1} + k_{i+1})\cdots(n_{i+j-1} + k_{i+j-1})]} \underbrace{f(\cdot), ..., f(\cdot)}_{\cdots}, ...,  \},$$
i.e.,
\begin{eqnarray*}
f \mapsto \mathrm{diag}\{\underbrace{f\circ \pi_1, ..., f\circ\pi_{c_i \cdots c_{i+j-1}}}_{n_i \cdots n_{i+j-1}},  f(E_i)1_{(n_{i+1} + k_{i+1})\cdots(n_{i+j-1} + k_{i+j-1})}, \\ (f(\pi_1(E_{i+1})), ..., f(\pi_{n_i}(E_{i+1})))1_{(n_{i+2} + k_{i+2})\cdots(n_{i+j-1} + k_{i+j-1})}, ...,  \}.
\end{eqnarray*}
So, it can be described as $$\mathrm{diag}\{\underbrace{f\circ \pi_1, ..., f\circ\pi_{c_i \cdots c_{i+j-1}}}_{n_i \cdots n_{i+j-1}}, \textrm{point evaluations}\}.$$

We shall choose $c_i$, $s_{i, 1}, ..., s_{i, c_i}$ (hence $n_i$), and $k_i$ in such a way 
$$ \lim_{j\to\infty} \frac{n_i \cdots n_{i+j}}{(n_i + k_i) \cdots (n_{i+j} + k_{i+j})} = \lim_{j\to\infty} (\frac{n_i}{n_i+k_i}) \cdots (\frac{n_{i+j}}{n_{i+j}+k_{i+j}}) \neq 0.$$
In other words, we require 
\begin{equation}\label{rdg-cond}
 \lim_{i\to\infty} \lim_{j\to\infty} \frac{n_i \cdots n_{i+j}}{(n_i + k_i) \cdots (n_{i+j} + k_{i+j})} = \lim_{i\to\infty}\lim_{j\to\infty} (\frac{n_i}{n_i+k_i}) \cdots (\frac{n_{i+j}}{n_{i+j}+k_{i+j}}) = 1.
 \end{equation}
 Denote the inductive limit algebra by
 $$A(X, (n_i), (k_i), E).$$
In what follows, we shall show, with a mild assumption on $X$ (see Definition \ref{solid-sp}), that this algebra, which is always simple, is independent of the choice of points in the point-evaluation set $E$ (if the number of them at each stage is kept the same; otherwise, it is classified by the $\Kzero$-group together with the radius of comparison). In the case that $X$ is contractible, we shall show that this C*-algebra is classified by the $\Kzero$-group and the radius of comparison, and also, if the latter is zero, the trace simplex (Theorem \ref{n-theorm}). (Even allowing different numbers of coordinate projections and multiplicities, and different numbers of point evaluations.)

\begin{rem}
If $c_i=1$, $i=1, 2, ...$, then $A(X, (n_i), (k_i), E)$ is the C*-algebra constructed by Goodearl in \cite{Goodearl-AH} with real rank not equal to zero. On the other hand, if $s_{i, j}=1$, $i=1, 2, ...,$ $j=1, ..., c_i$, then $A(X, (n_i), (k_i), E)$ is the C*-algebra constructed by Villadsen in \cite{Vill-perf}.
\end{rem}

\section{Mean dimension and radius of comparison}

In this section, let us calculate the mean dimension (as formulated in \cite{Niu-MD}) and radius of comparison (as formulated in \cite{RC-Toms}) of the Goodearl--Villadsen algebras $A(X, (n_i), (k_i), E)$. 

First, recall 
\begin{defn}[Definition 6.1 of \cite{RC-Toms}]\label{RC}
Let $A$ be a C*-algebra. Denote by $\mathrm{M}_n(A)$ the C*-algebra of $n\times n$ matrices over $A$. Regard $\mathrm{M}_n(A)$ as the upper-left corner of $\mathrm{M}_{n+1}(A)$, and consider the union, $$\mathrm{M}_\infty(A) = \bigcup_{n=1}^\infty \mathrm{M}_n(A),$$ the algebra of all finite matrices over $A$.

The radius of comparison of a unital C*-algebra $A$, denoted by $\mathrm{rc}(A)$, is the infimum of the set of real numbers $r > 0$ such that  if $a, b\in (\mathrm{M}_\infty(A))^+$ satisfy
$$\mathrm{d}_\tau(a) + r < \mathrm{d}_\tau(b),\quad \tau\in\mathrm{T}(A),$$ then $a \precsim b$, where $\mathrm{T}(A)$ is the simplex of tracial states. (In \cite{RC-Toms}, the radius of comparison is defined in terms of quasitraces instead of traces; but since all the algebras considered in this paper are nuclear, by \cite{Haagtrace} (see also \cite{BW-N} in the nuclear case), any quasitrace is actually a trace.)
\end{defn}

We also have the following remark on vector bundles:
\begin{rem}\label{ob-chern}
Assume a (complex) vector bundle $E$ over a compact metrizable space $X$ has non-zero Chern class $\mathrm{c}_n(E) \in \mathrm{H}^{n}(X)$. Then the trivial sub-bundles of $E$ have rank at most $\mathrm{rank}(E) - n/2$, as,  if there is a trivial sub-bundle $F$ of rank $r> \mathrm{rank}(E) - n/2$, then 
$$\mathrm{c}(E) = \mathrm{c}(F^c\oplus F) = \mathrm{c}(F^c)\mathrm{c}(F) = \mathrm{c}(F^c),$$
but, since $\mathrm{rank}(F^c)=\mathrm{rank}(E) - \mathrm{rank}(F) < n/2,$ we have $\mathrm{c}_n(F^c) = 0$, and hence $\mathrm{c}_n(E) = \mathrm{c}_n(F^c) =0$, which contradicts the assumption.
\end{rem}

\begin{defn}\label{solid-sp}
Let us call a metrizable compact space $X$ solid if it contains a Euclidean ball of dimension $\mathrm{dim}(X)$ when $\mathrm{dim}(X)$ is finite; when $\mathrm{dim}(X) = \infty$, $X$ solid will mean that $X$ contains a  Euclidean ball of arbitrarily large dimension.
\end{defn}

Note that all finite CW-complexes are solid. The Hawaiian earring and the Hilbert cube are also solid. Not all compact metrizable spaces are solid, as there are such $X$ with $\mathrm{dim}(X \times X) < 2 \cdot  \mathrm{dim}(X)$ which implies that $X$ cannot be solid.

\begin{thm}\label{rcA}
Let $X$ be a metrizable compact space. With $A = A(X, (n_i), (k_i), E)$, one has 
\begin{equation}\label{bound-mdim}
\mathrm{mdim}(A) \leq \frac{\dim(X)}{n_0} \cdot \lim_{i\to\infty}\frac{c_1 \cdots c_i}{(n_1+k_1)\cdots (n_i+k_i)},
\end{equation} 
where $\infty \cdot 0 = 0$ and $\mathrm{mdim}(\cdot)$ is the mean dimension of an AH system introduced in \cite{Niu-MD}.

Moreover, if the metrizable compact  space $X$ is solid, then equality holds in \eqref{bound-mdim}, and the radius of comparison of $A$, $\mathrm{rc}(A)$, is equal to $\frac{1}{2} \mathrm{mdim}(A)$.
\end{thm}

\begin{proof}
Let us first prove \eqref{bound-mdim}. Consider the $i$th stage, $\mathrm{M}_{m_i}(\mathrm{C}(X^{d_i}))$, where $m_i=n_0(n_1+k_1)\cdots(n_{i-1}+k_{i-1})$ and $d_i=c_1 \cdots c_{i-1}$ (note that all coordinate projections appear, i.e., $s_{i, 1}, ..., s_{i, c_i} \geq 1$, $i=1, 2, ...$), and let $\alpha$ be a finite open cover of $X^{d_i}$. Since the pull-back of $\alpha$ by any constant map has degree zero and $\mathcal D(\alpha) \leq \mathrm{dim}(X^{d_i}) \leq c_1\cdots c_{i-1} \cdot \mathrm{dim}(X)$, we have
$$\mathcal D(\phi_{i, j}(\alpha)) \leq c_i\cdots c_{j}\mathcal \cdot \mathcal D(\alpha) \leq c_1\cdots c_j \cdot \mathrm{dim}(X),$$
where $\mathcal D(\cdot)$ denotes the degree of an open cover,
and then
\begin{eqnarray*}
\lim_{j\to\infty} \frac{\mathcal D(\phi_{i, j}(\alpha))}{m_{j+1}} & \leq & \lim_{j\to\infty} \frac{ c_i\cdots c_{j}\mathcal \cdot \mathcal D(\alpha)}{n_0(n_1+k_1)\cdots (n_j+k_j)}\\
&  = & \frac{\mathcal D(\alpha)}{n_0} \cdot \lim_{j\to\infty}\frac{c_i\cdots c_j}{ (n_1+k_1)\cdots (n_j+k_j) } \\
& \leq & \frac{\mathrm{dim}(X)}{n_0} \cdot \lim_{j\to\infty}\frac{c_1\cdots c_j}{ (n_1+k_1)\cdots (n_j+k_j) },
\end{eqnarray*}
where $\infty \cdot 0 = 0$ in the case that $\mathrm{dim}(X) = \infty$. Passing to the limit as $i\to\infty$, we obtain \eqref{bound-mdim} (with $i$ in the place of $j$). In particular, by \cite{Niu-MD}, we have
\begin{equation}\label{rc-mdim-pre}
\mathrm{rc}(A) \leq \frac{1}{2}\mathrm{mdim}(A) \leq \frac{1}{2} \cdot \frac{\mathrm{dim}(X)}{n_0} \cdot \lim_{i\to\infty}\frac{c_1 \cdots c_i}{(n_1+k_1)\cdots (n_i+k_i)}. 
\end{equation}

Now, assume $X$ is solid (i.e., it contains a Euclidean ball of dimension $\mathrm{dim}(X)$, if $\mathrm{dim}(X) < \infty$, and of arbitrary dimension otherwise), and let us show that 
\begin{equation}\label{bound-rc}
\mathrm{rc}(A) \geq \frac{1}{2} \cdot \frac{\mathrm{dim}(X)}{n_0} \cdot \lim_{i\to\infty}\frac{c_1 \cdots c_i}{(n_1+k_1)\cdots (n_i+k_i)}. 
\end{equation} 
Together with \eqref{rc-mdim-pre}, we will then have $$\mathrm{rc}(A) = \frac{1}{2} \mathrm{mdim}(A) = \frac{1}{2}\cdot \frac{\dim(X)}{n_0} \cdot \lim_{i\to\infty}\frac{c_1 \cdots c_i}{(n_1+k_1)\cdots (n_i+k_i)}.$$

 Set $$\gamma= \lim_{i\to\infty} \frac{c_1 \cdots c_i}{(n_1+k_1)\cdots (n_i+k_i)}.$$ Since \eqref{bound-rc} holds trivially if $\gamma=0$ (as $\infty \cdot 0 = 0$), let us assume that $\gamma \neq 0$ in the rest of the proof.

Suppose that $\mathrm{dim}(X)<\infty$.
Let $\eps>0$ be arbitrary for the time being. 
Choose $i$ sufficiently large that
$$ \frac{c_1\cdots c_{i-1} \cdot \mathrm{dim}(X) - 2}{2n_0(n_1+k_1)\cdots(n_{i-1}+k_{i-1})} >  \frac{\gamma}{2} \cdot \frac{\mathrm{dim}(X)}{n_0} - \eps$$
and
$$ \frac{\mathrm{dim}(X)}{2n_0}\left(\frac{c_1\cdots c_{i-1}}{(n_1+k_1)\cdots(n_{i-1}+k_{i-1})} - \frac{c_{1}\cdots c_{j-1}}{(n_1+k_1)\cdots(n_{j-1}+k_{j-1})}\right)  <\eps,\quad j> i.$$

Since $X$ contains a Euclidean ball of dimension $\mathrm{dim}(X)$, the space $X^{c_1\cdots c_{i-1}}$ contains a Euclidian ball of dimension $c_1\cdots c_{i-1} \cdot \mathrm{dim}(X)$, and hence, if $i$ is large enough, it contains a $d$-dimensional sphere $S$, where $$c_1\cdots c_{i-1} \cdot \mathrm{dim}(X)-2 \leq d \leq c_1\cdots c_{i-1} \cdot \mathrm{dim}(X) - 1$$ and $d$ is non-zero and even.

Pick a (complex) vector bundle $E$ over $S$ such that $\mathrm{rank}(E)=d/2$ and $e:=\mathrm{c}_d(E)\in \mathrm{H}^d(S)$ is non-zero, where $\mathrm{c}_d$ is the $d$th Chern class. (Such a vector bundle exists, as, otherwise, the Chern class of every vector bundle would be trivial, then the Chern character is trivial, and it would not induce a rational isomorphism between the K-group and the cohomology group of the sphere $S$.) 
Note that the total Chern class of $E$ is $1 + e$. Denote by $p$ the corresponding projection in $\mathrm{M}_\infty(\mathrm{C}(S))$, and extend $p$ to a positive element of $\mathrm{M}_\infty(\mathrm{C}(X^{d_{i}}))$ such that $\mathrm{rank}(p(x)) \geq d/2$, $x\in X^{d_{i}}$. Denote this element still by $p$. 

Note that
$$\mathrm{d}_\tau(p) \geq \frac{d}{2n_0(n_1+k_1)\cdots(n_{i-1}+k_{i-1})} \geq \frac{c_1\cdots c_{i-1} \cdot \mathrm{dim}(X) - 2}{2n_0(n_1+k_1)\cdots(n_{i-1}+k_{i-1})} >  \frac{\gamma}{2} \cdot \frac{\mathrm{dim}(X)}{n_0} - \eps.$$

Consider the element $\phi_{i, \infty}(p)\in A$. For each $j> i$, the restriction of $\phi_{i, j}(p)\in \mathrm{M}_{m_j}(\mathrm{C}(X^{d_j}))$ to $S\times\cdots\times S \subseteq X^{d_j}$ is a projection which corresponds to the vector bundle $$E_j:=(\bigoplus_{s_1} \pi_1^*(E)) \oplus \cdots \oplus (\bigoplus_{s_{c_i\cdots c_{j-1}}} \pi_{c_i\cdots c_{j-1}}^*(E)) \oplus \theta_{j},$$
where $\theta_j$ is a trivial bundle. Then the total Chern class of $E_j$ is 
\begin{eqnarray*}
& & \pi_1^*(1 + c_d)^{s_1}\pi_2^*(1 + c_d)^{s_2}\cdots \pi_{c_{i}\cdots c_{j-1}}^*(1 + c_d)^{s_{c_{i}\cdots c_{j-1}}} \\
& = & \pi_1^*(1 + s_1e) \pi_2^*(1 + s_2 e) \cdots \pi_{c_{i}\cdots c_{j-1}}^*(1 + s_{c_{i}\cdots c_{j-1}}e),
\end{eqnarray*}
and, by the K\"unneth Theorem, it is non-zero at degree $dc_{i}\cdots c_{j-1}$. Hence (see Remark \ref{ob-chern}),  any trivial sub-bundle of $E_j$ has rank at most 
\begin{eqnarray*}
&& \mathrm{rank}(E_j) - \frac{1}{2} dc_{i}\cdots c_{j-1} \\
& = & \mathrm{rank}(E)(n_{i}+k_{i})\cdots(n_{j-1}+k_{j-1}) - \frac{1}{2} dc_{i}\cdots c_{j-1} \\ 
& = & \frac{d}{2} ( (n_{i}+k_{i})\cdots(n_{j-1}+k_{j-1})  - c_{i}\cdots c_{j-1} )\\
& \leq &\frac{\mathrm{dim}(X)}{2}(c_1\cdots c_{i-1} (n_{i}+k_{i})\cdots(n_{j-1}+k_{j-1}) - c_{1}\cdots c_{j-1}) \\
& = &\frac{\mathrm{dim}(X)}{2n_0}(\frac{c_1\cdots c_{i-1}}{(n_1+k_1)\cdots(n_{i-1}+k_{i-1})} - \frac{c_{1}\cdots c_{j-1}}{(n_1+k_1)\cdots(n_{j-1}+k_{j-1})}) n_0(n_1+k_1)\cdots(n_{j-1}+k_{j-1}) \\
&\leq&\eps n_0 (n_1+k_1)\cdots(n_{j-1}+k_{j-1}).
\end{eqnarray*}

Let $r\in A$ be a trivial projection with $2\eps <\mathrm{d}_\tau(r) < 3\eps$. Then $$\mathrm{d}_\tau(r) + (\frac{\gamma}{2}\cdot \frac{\mathrm{dim}(X)}{n_0}-4\eps) < \mathrm{d}_\tau(p),\quad \tau\in\mathrm{T}(A).$$
But the rank of the vector bundle of $r$ at the stage $j$ is at least
$$2\eps n_0(n_1+k_1)\cdots(n_{j-1}+k_{j-1}) >  \eps n_0 (n_1+k_1)\cdots(n_{j-1}+k_{j-1}),$$
which implies that $r$ is not Cuntz subequivalent to $p$, and therefore, $$\mathrm{rc}(A)\geq \frac{\gamma}{2} \cdot \frac{\mathrm{dim}(X)}{n_0} - 4\eps.$$
Since $\eps$ is arbitrary, this implies $\mathrm{rc}(A)\geq \frac{\gamma}{2} \cdot \frac{\mathrm{dim}(X)}{n_0}$.

If $X$ is infinite-dimensional (recall still $\lim_{i\to\infty} \frac{c_1 \cdots c_i}{(n_1+k_1)\cdots (n_i+k_i)} = \gamma \neq 0$), then the argument above (choose $d$ arbitrarily large) shows that $\mathrm{rc}(A)$ is arbitrarily large, and hence $\mathrm{rc}(A)=\infty$. So, \eqref{bound-rc} always holds, as desired.
\end{proof}

\begin{cor}
For any $r \in [0, +\infty]$, there is a Villadsen algebra $A$ such that $\mathrm{rc}(A) = r$.
\end{cor}

\begin{proof}
Let us assume that $r \in (0, +\infty)$. Pick a natural number $d$ such that $2r < d$, and consider $s:=2r/d \in (0, 1)$. Then pick a sequence of rational numbers $p_i/q_i \in (0, 1)$, $i=1, 2, ...$, such that $$\frac{p_1}{q_1} \cdot \frac{p_2}{q_2} \cdots = s.$$ Write $$n_i = p_i \quad\mathrm{and}\quad k_i = q_i - p_i,\quad i=1, 2, ... , $$
and we have
$$(\frac{n_1}{n_1 + k_1}) (\frac{n_2}{n_2 + k_2}) \cdots = s.$$
Let $A$ be a Villadsen algebra associated with $(n_i)$ and $(k_i)$ (and $c_i = n_i$, $i=1, 2, ...$, $n_0=1$) with the seed space $X = [0, 1]^d$, which is solid. Then it follows from Theorem \ref{rcA} that
$$\mathrm{rc} = \frac{1}{2} d (\frac{n_1}{n_1 + k_1}) (\frac{n_2}{n_2 + k_2}) \cdots = \frac{1}{2}ds = r.$$

If $r = +\infty$, then one can construct a Villadsen algebra with the seed space $X=[0, 1]^\infty$ and with the sequences $(n_i)$, $(k_i)$ as above. Then the resulting algebra has $\mathrm{rc}(A) = +\infty$.

If $r = 0$, then one can construct a Villadsen algebra with the seed space $X$ being the Cantor set and with the sequences $(n_i)$, $(k_i)$ as above. Then the resulting algebra has $\mathrm{rc}(A) = 0$.
\end{proof}

\begin{thm}\label{cor-almost-1}
If $\mathrm{rc}(A) > 0$, then
\begin{equation}\label{prod-m-1}
\lim_{i\to\infty}\lim_{j\to\infty} (\frac{c_i}{n_i}) \cdots (\frac{c_{i+j}}{n_{i+j}}) = 1 
\end{equation}
and  
\begin{equation}\label{almost-m-1}
\lim_{i\to\infty}\lim_{j\to\infty} \frac{\abs{\{s_k: s_k=1, k=1, ..., c_i\cdots c_{i+j}\}}}{n_i\cdots n_{i+j}} = 1,
\end{equation}
where
$$\phi_{i, j+1} = \mathrm{diag}\{ \underbrace{\underbrace{\pi_1^*, ..., \pi_1^*}_{s_1}, ..., \underbrace{\pi_{c_{i}\cdots c_{i+j}}^*, ..., \pi_{c_i\cdots c_{i+j}}^*}_{s_{c_i\cdots c_{i+j}}}}_{n_i\cdots n_{i+j}},\  \mathrm{point\  evaluations}\}.$$
\end{thm}
\begin{proof}
Since $\mathrm{rc}(A) > 0$, we have $$\lim_{i\to\infty}\frac{c_1 \cdots c_i}{(n_1+k_1)\cdots (n_i+k_i)} >0,$$ and hence
$$\lim_{i\to\infty} \lim_{j\to\infty}\frac{c_i \cdots c_{i+j}}{(n_i+k_i)\cdots (n_{i+j}+k_{i+j})} =1.$$
Comparing this with \eqref{rdg-cond} (and since both limits are non-zero), we have \eqref{prod-m-1}:
\begin{eqnarray*}
1& = & \lim_{i\to\infty}\frac{ \displaystyle \lim_{j\to\infty} \frac{c_i \cdots c_{i+j}}{(n_i+k_i)\cdots (n_{i+j}+k_{i+j})}  }{\displaystyle \lim_{j\to\infty}\frac{n_i \cdots n_{i+j}}{(n_i+k_i)\cdots (n_{i+j}+k_{i+j})} } = \lim_{i\to\infty}\lim_{j\to\infty}\frac{ \displaystyle \frac{c_i \cdots c_{i+j}}{(n_i+k_i)\cdots (n_{i+j}+k_{i+j})} }{\displaystyle \frac{n_i \cdots n_{i+j}}{(n_i+k_i)\cdots (n_{i+j}+k_{i+j})} } \\
& = & \lim_{i\to\infty}\lim_{j\to\infty} (\frac{c_i}{n_i}) \cdots (\frac{c_{i+j}}{n_{i+j}}).
\end{eqnarray*}

As for \eqref{almost-m-1}, note that $n_i\cdots n_{i+j} = s_1 + \cdots + s_{c_i\cdots c_{i+j}}$, and hence
$$\frac{c_i\cdots c_{i+j}}{n_i\cdots n_{i+j}} = \frac{c_i\cdots c_{i+j}}{s_1 + \cdots + s_{c_i\cdots c_{i+j}}}\leq \frac{c_i\cdots c_{i+j}}{(c_i\cdots c_{i+j} - b_{i, j}) + 2 b_{i, j} } = \frac{c_i\cdots c_{i+j}}{c_i\cdots c_{i+j} + b_{i, j} } \leq 1,$$ where
$$b_{i, j}:= \abs{\{s_k: s_k>1, k=1, ..., c_i\cdots c_{i+j}\}}.$$
Together with \eqref{prod-m-1}, this yields 
$$1 = \lim_{i\to\infty}\lim_{j\to\infty}  \frac{c_i\cdots c_{i+j}}{c_i\cdots c_{i+j} + b_{i, j} } = \lim_{i\to\infty}\lim_{j\to\infty}  \frac{1}{1 + \frac{b_{i, j}}{c_i\cdots c_{i+j}} };$$ therefore (note that $s_k$, $k=1, ..., c_i\cdots c_{i+j}$, are non-zero), $$1=\lim_{i\to\infty}\lim_{j\to\infty} \frac{c_i\cdots c_{i+j} - b_{i, j}}{c_i\cdots c_{i+j}} = \lim_{i\to\infty}\lim_{j\to\infty} \frac{\abs{\{s_k: s_k=1, k=1, ..., c_i\cdots c_{i+j}\}}}{c_i\cdots c_{i+j}} .$$
Using \eqref{prod-m-1} again, one obtains \eqref{almost-m-1}.
\end{proof}

\section{Intertwinings of trace simplexes}\label{trace-decomp}

\subsection{Trace simplex of the Villadsen algebra}

Let us first observe that, under Condition \eqref{rdg-cond}, the trace simplex of the Villadsen algebra is independent of (the number and the location of) the point evaluations.

Denote by $A$ the (non-simple) limit of the inductive sequence
\begin{displaymath}
\xymatrix{
\mathrm{M}_{n_0}(\mathrm{C}(X)) \ar[r] & \mathrm{M}_{n_0n_1}(\mathrm{C}(X^{c_1})) \ar[r] & \mathrm{M}_{n_0n_1n_2}(\mathrm{C}(X^{c_1c_2})) \ar[r] & \cdots,
}
\end{displaymath}
where the seed for the $i$th-stage map,
$$\phi_i: \mathrm{C}(X^{c_1 \cdots c_{i-1}}) \to \mathrm{M}_{n_i+k_i}(\mathrm{C}(X^{c_1 \cdots c_{i-1}c_{i}})),$$ is defined by
\begin{eqnarray*} 
f & \mapsto & \mathrm{diag}\{ \underbrace{\underbrace{f\circ \pi_1, ..., f\circ \pi_1}_{s_{i, 1}}, ... ,  \underbrace{f\circ \pi_{c_i}, ..., f\circ\pi_{c_i}}_{s_{i, c_i}} }_{n_i} \}.
\end{eqnarray*}

\begin{lem}\label{trace-int-1}
Let $A_E$ be a Villadsen algebra with point-evaluation set $E$ which satisfies Condition \eqref{rdg-cond}. Then 
$\mathrm{T}(A_E) \cong \mathrm{T}(A).$
\end{lem}

\begin{proof}
Let $\delta_1, \delta_2, ...$ be a decreasing sequence of strictly positive numbers with $ \sum_{n=1}^\infty \delta_n < 1.$
Identifying $\mathrm{Aff}_\Real(\mathrm{T}(\mathrm{M}_s(\mathrm{C}(Y)))) = \mathrm{C}_\Real(Y)$ for any compact metrizable space $Y$, and note that $(\pi_{i, i+j})^*$ is given by 
$$\mathrm{C}_\Real(X^{c_1\cdots c_i}) \ni h \mapsto \frac{1}{n_i\cdots n_{i+j-1}}( \underbrace{h \circ \pi_1 + \cdots + h\circ\pi_{c_i \cdots c_{i+j-1}}}_{n_i\cdots n_{i+j-1}})  \in  \mathrm{C}_\Real(X^{c_1\cdots c_{i+j-1}}).$$ 
Then a straightforward calculation shows that, for any $h\in \mathrm{C}_\Real(X^{c_1\cdots c_i})$ with $\norm{h}_\infty \leq 1$,
\begin{eqnarray*}
&& \norm{(\pi_{i, i+j})^*(h) - (\phi^{(E)}_{i, i+j})^*(h)}_\infty \\
& = & \| \frac{1}{n_i\cdots n_{i+j-1}}( \underbrace{h \circ \pi_1 + \cdots + h \circ\pi_{c_i \cdots c_{i+j-1}}}_{n_i\cdots n_{i+j-1}}) \\
 & &      - \frac{1}{(n_i+k_i)\cdots (n_{i+j-1} + k_{i+j-1})}( \underbrace{h \circ \pi_1 + \cdots + h \circ\pi_{c_i \cdots c_{i+j-1}}}_{n_i\cdots n_{i+j-1}} + \textrm{point evaluations}) \|_\infty \\
 & \leq & (\frac{1}{n_i\cdots n_{i+j-1}} - \frac{1}{(n_i+k_i)\cdots (n_{i+j-1} + k_{i+j-1})}) (n_i\cdots n_{i+j-1})  \\
 & & +1-\frac{n_i\cdots n_{i+j-1}}{(n_i+k_i)\cdots (n_{i+j-1} + k_{i+j-1})} \\
 & = & 2 (1-\frac{n_i\cdots n_{i+j-1}}{(n_i+k_i)\cdots (n_{i+j-1} + k_{i+j-1})} ),
\end{eqnarray*}
which, by Condition \eqref{rdg-cond}, is arbitrarily small if $i$ is sufficiently large. Therefore, there is a diagram
\begin{equation*}
\xymatrix{
\mathrm{C}_{\Real}(X) \ar[r]^{(\phi_{1, i_1}^{(E)})^*} \ar[dr]^{(\pi_{1, i_1})^*} & \mathrm{C}_{\Real}(X^{d_{i_1}}) \ar[r]^{(\phi_{i_1, i_2}^{(E)})^*}  & \mathrm{C}_{\Real}(X^{d_{i_2}}) \ar[r]  & \cdots \ar[r] & (\mathrm{Aff}_\Real(\mathrm{T}(A_E)), \norm{\cdot}_\infty) \\
\mathrm{C}_{\Real}(X) \ar[r]^{(\pi_{1, i_1})^*}  & \mathrm{C}_{\Real}(X^{d_{i_1}})\ar[r]^{(\pi_{i_1, i_2})^*} \ar[ur]^{(\pi_{i_1, i_2})^*}  & \mathrm{C}_{\Real}(X^{d_{i_2}}) \ar[r] & \cdots \ar[r] & (\mathrm{Aff}_\Real(\mathrm{T}(A)), \norm{\cdot}_\infty)
}
\end{equation*}
with
$$\norm{(\pi_{i_{s+1}, i_{s+2}})^* \circ (\pi_{i_s, i_{s+1}})^*(h) - \phi^{(E)}_{i_{s+1}, i_{s+2}} \circ \phi^{(E)}_{i_s, i_{s+1}}(h)}_\infty < \delta_s$$
for any $s=0, 2, ...$, any $h\in \mathrm{C}_\Real(X^{d_{i_s}})$ with $\norm{h}_\infty \leq 1$. 
This implies in particular that $\mathrm{T}(A_E) \cong \mathrm{T}(A).$
\end{proof}

Let us calculate the trace simplex of $A$. Note that $\mathrm{T}(A)$ is homeomorphic to the limit of the following affine projective system:
\begin{displaymath}
\xymatrix{
\mathcal{M}_1(X) & \mathcal{M}_1(X^{c_1}) \ar[l] & \mathcal{M}_1(X^{c_1c_2}) \ar[l] & \cdots \ar[l] 
}
\end{displaymath}
where $\mathcal M_1(\cdot)$ denotes the simplex of probability measures and the connecting map $\theta_i: \mathcal{M}_1(X^{c_1\cdots c_{i}}) \to \mathcal{M}_1(X^{c_1\cdots c_{i-1}})$ is given by
$$\theta_i(\delta_{(x_1, ..., x_{c_i})}) = \frac{1}{n_i}(\underbrace{\delta_{x_1} + \cdots + \delta_{x_1}}_{s_{i, 1}} + \cdots + \underbrace{\delta_{x_{c_i}} + \cdots + \delta_{x_{c_i}}}_{s_{i, c_i}}),\quad x_1, ..., x_{c_i} \in X^{c_1\cdots c_{i-1}},$$
where $\delta_{x}$ denotes the Dirac measure concentrating at $x$.

The following lemma is a simple observation:
\begin{lem}\label{extrem}
Let $\tau = (\mu_i)$ be a tracial state of $A$, where $\mu_i$, $i=1, 2, ...$, is a probability measure of $X^{c_1\cdots c_{i-1}}$. If $\mu_i$ are Dirac measures for sufficiently large $i$, then $\tau$ is extreme. 
\end{lem}

\begin{proof}
Assume $$(\mu_i) = \alpha(\nu^{(1)}_i) + (1-\alpha)(\nu^{(2)}_i)$$ for some $\alpha\in(0, 1)$, where $\nu^{(1)}_i$ and $\nu^{(2)}_i$ are probablity measures of $X^{c_1\cdots c_{i-1}}$.  Since $\mu_i$ is extreme for sufficiently large $i$,  we have that $$\nu_i^{(1)} = \nu_i^{(2)} = \mu_i$$ for sufficiently large $i$, and hence $ (\nu_i^{(1)}) = (\nu_i^{(2)}) = (\mu_i)$, as desired.
\end{proof}

\begin{rem}
Note that, since the multiplicity of the coordinate projections are nonzero ($s_{i, j} \neq 0$), if $\theta_i(\mu)$ is a Dirac measure, then $\mu$ must be a Dirac measure. 
\end{rem}

Pick a point $x = (x_1, ..., x_{c_1\cdots c_{i-1}}) \in X^{c_1\cdots c_{i-1}}$, and then consider the trace $\tau_x$ of $A$ which is defined by
$$\tau_x=(\theta_{1, i}(\delta_x), ... , \theta_{i-1, i}(\delta_x),  \delta_x, \delta_{(\underbrace{x, ..., x}_{c_i})}, ..., \delta_{(\underbrace{x, ..., x}_{c_i\cdots c_{i+k}})} ...).$$
By the lemma above, $\tau_x$ is an extreme trace. Also note that if $x \neq y$, then $\tau_x \neq \tau_y$. Hence if the seed space $X$ is not a singleton, the trace simplex is not a singleton.

The following lemma is a direct consequence of the Krein-Milman Theorem.
\begin{lem}\label{discrete}
Let $\mathcal F \subseteq \mathrm{C}(X)$ be a finite set, let $\mu \in \mathcal{M}_1(X)$, and let $\eps>0$. Then, there is $N \in \mathbb N$ such that for any $n > N$, there are $x_1, ..., x_n \in X$ such that
$$\abs{\mu(f) - \frac{1}{n}(f(x_1) + \cdots + f(x_n))} < \eps,\quad f\in\mathcal F.$$
\end{lem}

\begin{thm}\label{trace-VA}
Assume 
\begin{equation}\label{cn-same}
\lim_{i\to\infty}\lim_{j\to\infty} (\frac{c_i}{n_i}) \cdots (\frac{c_{i+j}}{n_{i+j}}) = 1.
\end{equation}
Then the extreme points of $\mathrm{T}(A)$ are dense, i.e., $\mathrm{T}(A)$ is the Poulsen simplex if $X$ is not a singleton (\cite{Poulsen-Simplex}). 

The trace simplex of the simple Villadsen algebra $A_E$ with non-zero radius of comparison is the Poulsen simplex. 
\end{thm}

\begin{proof}
It is enough to show that the traces $\tau_x$ are dense. Let $\mu$ be a tracial state of $A$, and represent it as $$\mu=(\mu_1, \mu_2, ... ),$$ where $\mu_i$ is a probability measure of $X^{c_1 \cdots c_{i-1}}$, and $\theta_i(\mu_{i+1}) = \mu_{i}$, $i=1, 2, ...$.

Let $N(\mathcal F; \eps)$ be a neighborhood of $\mu$ given by 
$$ \{ \tau\in \mathrm{T}(A): \abs{\mu(f) - \tau(f)} < \eps,\ f \in \mathcal F \},$$
where $\mathcal F \in A$ is a finite set and $\eps>0$, and let us show that $\tau_x \in N$ for some $x \in X^{c_1\cdots c_{i-1}}$, $i\in \mathbb N$. Then the first statement of the theorem follows from Lemma \ref{extrem}.

By \eqref{cn-same} and the proof of Theorem \ref{cor-almost-1}, there is $i_0>0$ such that for all $j>0$,
\begin{equation}\label{close-c-n}
1 - \frac{c_{i_0} \cdots c_{i_0+j}}{n_{i_0} \cdots n_{i_0+j}} < \frac{\eps}{3}
\end{equation}
and
\begin{equation}\label{small-2-m}
 \frac{\abs{\{s_k: s_k=1, k=1, ..., c_{i_0}\cdots c_{i_0+j}\}}}{n_{i_0}\cdots n_{i_0+j}} > 1-\frac{\eps}{3},
\end{equation}
where
$$\phi_{i_0, j+1} = \mathrm{diag}\{ \underbrace{\underbrace{\pi_1^*, ..., \pi_1^*}_{s_1}, ..., \underbrace{\pi_{c_{i_0}\cdots c_{i_0+j}}^*, ..., \pi_{c_{i_0}\cdots c_{i_0+j}}^*}_{s_{c_{i_0}\cdots c_{i_0+j}}}}_{n_{i_0}\cdots n_{i_0+j}}\}.$$

Without loss of generality, we may assume that $\mathcal F$ is in the unit ball of $\mathrm{M}_{n_0\cdots n_{i_0-1}}(\mathrm{C}(X^{c_1\cdots c_{i_0-1}}))$. Then consider the measure $\mu_{i_0}$. By Lemma \ref{discrete}, there is a large enough $j_0$ that $c_{i_0}\cdots c_{i_0+j_0}$ is large enough that there are $$x_1, ..., x_{c_{i_0}\cdots c_{i_0+j_0}} \in X^{c_1\cdots c_{i_0-1}}$$ satisfying
$$\abs{ \mu_{i_0}(f) - \frac{1}{c_{i_0} \cdots c_{i_0+j_0}}(f(x_1) + f(x_2) + \cdots + f(x_{c_{i_0}\cdots c_{i_0+j_0}}))} < \frac{\eps}{3}, \quad f\in \mathcal F.$$
Consider the point
$$x_\mu:= (x_1, x_2, ..., x_{c_{i_0}\cdots c_{i_0+j_0}}) \in (X^{c_1\cdots c_{i_0-1}})^{c_{i_0}\cdots c_{i_0+j_0}} = X^{c_1\cdots c_{i_0+j_0}},$$ and consider the trace $\tau_{x_\mu} \in \mathrm{T}(A)$. Then, for each $f\in\mathcal F$,
\begin{eqnarray*}
&& \abs{\mu(f) - \tau_{x_\mu}(f)} \\
& = & \abs{\mu_{i_0}(f) - \frac{1}{n_{i_0}\cdots n_{i_0+j_0}}(\underbrace{\delta_{x_1} + \cdots + \delta_{x_1}}_{s_1} + \cdots + \underbrace{\delta_{x_{c_{i_0}\cdots c_{i_0+j_0}}} + \cdots + \delta_{x_{c_{i_0}\cdots c_{i_0+j_0}}} }_{s_{c_{i_0}\cdots c_{i_0+j_0}}})(f)} \\
& < & \abs{\mu_{i_0}(f) - \frac{1}{n_{i_0}\cdots n_{i_0+j_0}}(\delta_{x_1} + \cdots + \delta_{x_{c_{i_0}\cdots c_{i_0+j_0}}})(f)} + \frac{\eps}{3} \quad\quad \textrm{(by \eqref{small-2-m})} \\
& < &  \abs{\mu_{i_0}(f) - \frac{1}{c_{i_0}\cdots c_{i_0+j_0}}(\delta_{x_1} + \cdots + \delta_{x_{c_{i_0}\cdots c_{i_0+j_0}}})(f)} + \frac{\eps}{3} + \frac{\eps}{3} \quad\quad\mathrm{(by \eqref{close-c-n})}  \\
& < & \eps,
\end{eqnarray*}
and this shows the first statement of the theorem.

Now, let $A_E$ be a simple Villadsen algebra with non-zero radius of comparison.  By Theorem \ref{cor-almost-1}, Equation \eqref{cn-same} holds (and $X$ is not a singleton), and therefore, $\mathrm{T}(A)$ is the Poulsen simplex. By Lemma \ref{trace-int-1}, $\mathrm{T}(A_E) \cong \mathrm{T}(A)$, and hence $\mathrm{T}(A_E)$ is the Poulsen simplex as well.
\end{proof}

\begin{rem}
Note that $\mathrm{T}(A_E)$ is always the Poulsen simplex whenever  \eqref{cn-same} holds and $X$ is not a singleton. This includes the case that $\mathrm{dim}(X) = 0$ (hence $\mathrm{rc}(A_E) = 0$).
\end{rem}

\begin{rem}
Note that the Giol-Kerr system, a dynamical analog of the Villadsen algebra,  is constructed as a small perturbation of the non-trivial two-sided shift (\cite{GK-Dyn}). It is worth comparing Theorem \ref{trace-VA} to the well-known fact that the simplex of invariant probability measures of the non-trivial two-sided shift is affinely homeomorphic to the Poulsen simplex (\cite{Sigmund-Poulsen}, \cite{Sigmund-Poulsen-2}).
\end{rem}


\subsection{An intertwining}

In what follows, we shall show (further to Lemma \ref{trace-int-1}) that for two point-evaluation sets, the intertwining maps between the trace simplices actually can be chosen to be induced by C*-algebra homomorphisms between building blocks, and in such a way that the resulting diagram of building blocks commutes (approximately) up to point evaluations.

Let there be given two different evaluation sets $$E_1, E_2, ..., E_i, ...\quad\mathrm{and}\quad F_1, F_2, ..., F_i, ..., $$ with sizes $(k^E_i)$ and $(k^F_i)$ respectively, and both satisfying Condition \eqref{rdg-cond} (with respect to the same $(n_i)$), and assume that, as supernatural numbers, 
\begin{equation}\label{K0-cond}
\prod_{i=1}^\infty (n_i+k_i^{(E)})=\prod_{i=1}^\infty (n_i+k_i^{(F)}),
\end{equation} 
and as real numbers, 
\begin{equation}\label{close-0}
\lim_{i\to\infty} \frac{ (n_1 + k_1^{(E)}) \cdots (n_{i} + k_{i}^{(E)}) }{ (n_1 + k_1^{(F)}) \cdots (n_{i} + k_{i}^{(F)}) } = 1.
\end{equation}

\begin{lem}\label{trace-int}
With the assumptions \eqref{K0-cond} and \eqref{close-0} above, let $A_E$ and $A_F$ denote the C*-algebras $A(X, (n_i), (k^{(E)}_i), E)$ and $A(X, (n_i), (k^{(F)}_i), F)$, respectively.
Let $\delta_1, \delta_2, ...$ be a decreasing sequence of strictly positive numbers with $$ \sum_{n=1}^\infty \delta_n < 1.$$
Then there is diagram
\begin{equation}\label{diag-0}
\xymatrix{
\mathrm{M}_{n_0}(\mathrm{C}(X)) \ar[r]^-{\phi_{1, i_1}^{(E)}} \ar[dr]^-{\phi_{1, i_1}^{(E, F)}} & \mathrm{M}_{m^{(E)}_{i_1}}(\mathrm{C}(X^{d_{i_1}})) \ar[r]^-{\phi_{i_1, i_2}^{(E)}}  & \mathrm{M}_{m^{(E)}_{i_2}}(\mathrm{C}(X^{d_{i_2}})) \ar[r] \ar[dr]^-{\phi_{i_2, i_3}^{(E, F)}} & \cdots \ar[r] & A_E \\
\mathrm{M}_{n_0}(\mathrm{C}(X)) \ar[r]^-{\phi_{1, i_1}^{(F)}}  & \mathrm{M}_{m^{(F)}_{i_1}}(\mathrm{C}(X^{d_{i_1}})) \ar[r]^-{\phi_{i_1, i_2}^{(F)}} \ar[ur]^-{\phi_{i_1, i_2}^{(F, E)}}  & \mathrm{M}_{m^{(F)}_{i_2}}(\mathrm{C}(X^{d_{i_2}})) \ar[r] & \cdots \ar[r] & A_F
}
\end{equation}
with
$$\abs{\tau(\phi^{(F, E)}_{i_{s+1}, i_{s+2}} \circ \phi^{(E, F)}_{i_s, i_{s+1}}(h) - \phi^{(E)}_{i_{s+1}, i_{s+2}} \circ \phi^{(E)}_{i_s, i_{s+1}}(h))} < \delta_s$$
for any $s=0, 2, ...$, any $h\in \mathrm{M}_{m^{(E)}_{i_s}}(\mathrm{C}(X^{d_{i_s}}))$ with $\norm{h}\leq 1$, and any $\tau\in\mathrm{T}(\mathrm{M}_{m^{(E)}_{i_{s+2}}}(\mathrm{C}(X^{d_{i_{s+2}}})))$, and, furthermore, 
$$\abs{\tau(\phi^{(E, F)}_{i_{s+1}, i_{s+2}} \circ \phi^{(F, E)}_{i_s, i_{s+1}}(h) - \phi^{(F)}_{i_{s+1}, i_{s+2}} \circ \phi^{(F)}_{i_s, i_{s+1}}(h))} < \delta_s$$
for any $s=1, 3, ...$, any $h\in \mathrm{M}_{m^{(F)}_{i_s}}(\mathrm{C}(X^{d_{i_s}}))$ with $\norm{h}\leq 1$, and any $\tau\in\mathrm{T}(\mathrm{M}_{m^{(F)}_{i_{s+2}}}(\mathrm{C}(X^{d_{i_{s+2}}})))$. 
Moreover, for each $s = 0, 2, ... $,
$$\phi^{(F, E)}_{i_{s+1}, i_{s+2}} \circ \phi^{(E, F)}_{i_s, i_{s+1}} =\mathrm{diag}\{\pi_1^*, ..., \pi_{n_{i_s}\cdots n_{i_{s+1}}}^*, \textrm{point evaluations} \}, $$ 
and
for each $s = 1, 3, ...$,
$$\phi^{(E, F)}_{i_{s+1}, i_{s+2}} \circ \phi^{(F, E)}_{i_s, i_{s+1}} =\mathrm{diag}\{\pi_1^*, ..., \pi_{n_{i_s}\cdots n_{i_{s+1}}}^*, \textrm{point evaluations} \}.$$ 
\end{lem}

\begin{defn}\label{close-size}
The sequences $(k_i^{(E)})$ and $(k_i^{(F)})$ will be said to be sufficiently close if 
for any $\delta>0$, there is an arbitrarily large pair $i_1>i'_1$ such that 
\begin{equation*}
1 - \prod_{j=0}^\infty \frac{n_{i'_1+j}}{n_{i'_1+j}+k^{(E)}_{i'_1+j}} < \delta,
\end{equation*}
\begin{equation*}
\textrm{$\prod_{i=1}^{i_1-1}(n_i+k_{i}^{(F)})$ is divisible by $\prod_{i=1}^{i'_1-1}(n_i+k_{i}^{(E)})$},
\end{equation*} 
and
\begin{equation}\label{close-1}
\frac{(n_1 + k_1^{(F)}) \cdots (n_{i_1'-1} + k_{i_1'-1}^{(F)})}{(n_1 + k_1^{(E)}) \cdots (n_{i_1'-1} + k_{i_1'-1}^{(E)})} \cdot \frac{(n_{i_1'} + k_{i_1'}^{(F)})\cdots (n_{i_1 - 1} + k_{i_1-1}^{(F)})}{n_{i_1'}\cdots n_{i_1-1}} > 1,
\end{equation}
and, furthermore, there are arbitrarily large $i_2>i'_2$ such that 
\begin{equation*}
1 - \prod_{j=0}^\infty \frac{n_{i'_2+j}}{n_{i'_2+j}+k^{(F)}_{i'_2+j}} < \delta,
\end{equation*}
\begin{equation*}
\textrm{$\prod_{i=1}^{i_2-1}(n_i+k_{i}^{(E)})$ is divisible by $\prod_{i=1}^{i'_2-1}(n_i+k_{i}^{(F)})$},
\end{equation*} 
and
\begin{equation}\label{close-2}
\frac{(n_1 + k_1^{(E)}) \cdots (n_{i_2'-1} + k_{i_2'-1}^{(E)})}{(n_1 + k_1^{(F)}) \cdots (n_{i_2'-1} + k_{i_2'-1}^{(F)})} \cdot \frac{(n_{i_2'} + k_{i_2'}^{(E)})\cdots (n_{i_2 - 1} + k_{i_2-1}^{(E)})}{n_{i_2'}\cdots n_{i_2-1}} > 1.
\end{equation}

\end{defn}

\begin{lem}\label{auto-close}
Under the assumptions \eqref{rdg-cond}, \eqref{K0-cond}, and \eqref{close-0}, the sequences $(k_i^{(E)})$ and $(k_i^{(F)})$ are sufficiently close. 
\end{lem}
\begin{proof}
We only have to show  \eqref{close-1} and \eqref{close-2}.
For the given $\delta>0$, choose $i_1'$ sufficiently large that 
\begin{equation*}
1 - \prod_{j=0}^\infty \frac{n_{i'_1+j}}{n_{i'_1+j}+k^{(E)}_{i'_1+j}} < \delta.
\end{equation*}

Then, with sufficiently large $i_1$, by \eqref{close-0}, we have
\begin{eqnarray*}
& & \frac{(n_1 + k_1^{(F)}) \cdots (n_{i_1'-1} + k_{i_1'-1}^{(F)})}{(n_1 + k_1^{(E)}) \cdots (n_{i_1'-1} + k_{i_1'-1}^{(E)})} \cdot \frac{(n_{i_1'} + k_{i_1'}^{(F)})\cdots (n_{i_1 - 1} + k_{i_1-1}^{(F)})}{n_{i_1'}\cdots n_{i_1-1}} \\
& = & \frac{(n_1 + k_1^{(F)}) \cdots (n_{i_1-1} + k_{i_1-1}^{(F)})}{(n_1 + k_1^{(E)}) \cdots (n_{i_1-1} + k_{i_1-1}^{(E)})} \cdot \frac{(n_{i_1'} + k_{i'_1}^{(E)})\cdots (n_{i_1 - 1} + k_{i_1-1}^{(E)})}{n_{i_1'}\cdots n_{i_1-1}} \\
& > & \frac{(n_1 + k_1^{(F)}) \cdots (n_{i_1-1} + k_{i_1-1}^{(F)})}{(n_1 + k_1^{(E)}) \cdots (n_{i_1-1} + k_{i_1-1}^{(E)})} \cdot \frac{n_{i_1'} + k_{i'_1}^{(E)}}{n_{i_1'}} >1.
\end{eqnarray*}
So \eqref{close-1} holds. A similar argument shows that \eqref{close-2} holds. 
\end{proof}

\begin{proof}[Proof of Lemma \ref{trace-int}]
Consider the inductive limit decompositions
$$
\xymatrix{
\mathrm{M}_{n_0}(\mathrm{C}(X)) \ar[r]^-{\phi_{1}^{(E)}} & \mathrm{M}_{n_0(n_1+k^{(E)}_1)}(\mathrm{C}(X^{c_1})) \ar[r]^-{\phi_{2}^{(E)}}  & \mathrm{M}_{n_0(n_1+k^{(E)}_1)(n_2+k^{(E)}_2)}(\mathrm{C}(X^{c_1c_2})) \ar[r] & \cdots \ar[r] & A_E, \\
\mathrm{M}_{n_0}(\mathrm{C}(X)) \ar[r]^-{\phi_{1}^{(F)}}  & \mathrm{M}_{n_0(n_1+k^{(F)}_1)}(\mathrm{C}(X^{c_1})) \ar[r]^-{\phi_{2}^{(F)}}  & \mathrm{M}_{n_0(n_1+k^{(F)}_1)(n_2+k^{(F)}_2)}(\mathrm{C}(X^{c_1c_2})) \ar[r] & \cdots \ar[r] & A_F.
}
$$

Since the sequences $(k_i^{(E)})$ and $(k_i^{(F)})$ are sufficiently close (Lemma \ref{auto-close}), there is a pair $i'_1<i_1$ sufficiently large that
\begin{equation}\label{small-E-1}
1 - \prod_{j=0}^\infty \frac{n_{i'_1+j}}{n_{i'_1+j}+k^{(E)}_{i'_1+j}} < \delta_1,
\end{equation}
\begin{equation}\label{divide-cond}
\textrm{$\prod_{i=1}^{i_1-1}(n_i+k_{i}^{(F)})$ is divisible by $\prod_{i=1}^{i'_1-1}(n_i+k_{i}^{(E)})$},
\end{equation} 
and 
\begin{equation}\label{enough-room}
\frac{(n_1 + k_1^{(F)}) \cdots (n_{i_1'-1} + k_{i_1'-1}^{(F)})}{(n_1 + k_1^{(E)}) \cdots (n_{i_1'-1} + k_{i_1'-1}^{(E)})} \cdot \frac{(n_{i_1'} + k_{i_1'}^{(F)})\cdots (n_{i_1 - 1} + k_{i_1-1}^{(F)})}{n_{i_1'}\cdots n_{i_1-1}} > 1.
\end{equation}

Then consider the diagram 
$$
\xymatrix{
\mathrm{M}_{n_0}(\mathrm{C}(X)) \ar[r]^-{\phi_{1, i'_1}^{(E)}}  & \mathrm{M}_{m^{(E)}_{i'_1}}(\mathrm{C}(X^{d_{i'_1}})) \ar[r]^-{\phi_{i'_1, i_1}^{(E)}} \ar[dr]^-{\tilde{\phi}_{i'_1, i_1}^{(E, F)}}  & \mathrm{M}_{m^{(E)}_{i_1}}(\mathrm{C}(X^{d_{i_1}})) \ar[r] & \cdots \ar[r] & A_E \\
\mathrm{M}_{n_0}(\mathrm{C}(X)) \ar[r]^-{\phi_{1, i'_1}^{(F)}}  & \mathrm{M}_{m^{(F)}_{i'_1}}(\mathrm{C}(X^{d_{i'_1}})) \ar[r]^-{\phi_{i'_1, i_1}^{(F)}}  & \mathrm{M}_{m^{(F)}_{i_1}}(\mathrm{C}(X^{d_{i_1}})) \ar[r] & \cdots \ar[r] & A_F,
}
$$
where 
\begin{equation}\label{m-and-d}
m_i:=n_0(n_1+k_1)\cdots (n_{i-1}+k_{i-1}), \quad d_i:=c_1\cdots c_{i-1},
\end{equation}
and
$\tilde{\phi}_{i'_1, i_1}^{(E, F)}: \mathrm{M}_{m^{(E)}_{i'_1}}(\mathrm{C}(X^{d_{i'_1}})) \to \mathrm{M}_{m^{(F)}_{i_1}}(\mathrm{C}(X^{d_{i_1}}))$ is a map
$$f \mapsto \mathrm{diag}\{\underbrace{f\circ \pi_1, ..., f\circ\pi_{c_{i_1'} \cdots c_{i_1-1}}}_{n_{i_1'} \cdots n_{i_1-1}}, \textrm{point evaluations}\},$$ where the point evaluations are arbitrarily chosen (by \eqref{divide-cond} and \eqref{enough-room}, there is enough room for the map $\tilde{\phi}_{i'_1, i_1}^{(E, F)}$ to exist) (the evaluation points chosen for the map $\tilde{\phi}_{i'_1, i_1}^{(E, F)}$ might not be suitably dense in $X^{d_{i'_1}}$, but the set of evaluation points of the map  $\tilde{\phi}_{i'_1, i_1}^{(E, F)} \circ \phi^{(E)}_{1, i_1'} $, which contains the set of evaluation points of the map $\phi^{(E)}_{1, i_1'}$, is suitably dense as $i_1'\to\infty$).

Write $$\phi_{1, i_1}^{(E, F)}=\tilde{\phi}_{i'_1, i_1}^{(E, F)} \circ \phi_{1, i'_1}^{(E)},$$ and compress the diagram above as
$$
\xymatrix{
\mathrm{M}_{n_0}(\mathrm{C}(X)) \ar[r]^-{\phi_{1, i_1}^{(E)}} \ar[dr]^{\phi_{1, i_1}^{(E, F)}} & \mathrm{M}_{m^{(E)}_{i_1}}(\mathrm{C}(X^{d_{i_1}})) \ar[r]^-{\phi_{i_1}^{(E)}}  & \mathrm{M}_{m^{(E)}_{i_1+1}}(\mathrm{C}(X^{d_{i_1+1}})) \ar[r] & \cdots \ar[r] & A_E \\
\mathrm{M}_{n_0}(\mathrm{C}(X)) \ar[r]^-{\phi_{1, i_1}^{(F)}}  & \mathrm{M}_{m^{(F)}_{i_1}}(\mathrm{C}(X^{d_{i_1}})) \ar[r]^-{\phi_{i_1}^{(F)}}  & \mathrm{M}_{m^{(F)}_{i_1+1}}(\mathrm{C}(X^{d_{i_1+1}})) \ar[r] & \cdots \ar[r] & A_F.
}
$$

There are $i'_2< i_2$ sufficiently large that 
$$1 - \prod_{j=0}^\infty \frac{n_{i'_2+j}}{n_{i'_2+j}+k^{(F)}_{i'_2+j}} < \delta_2,$$ 
\begin{equation*}
\textrm{$\prod_{i=1}^{i_2-1}(n_i+k_{i}^{(E)})$ is divisible by $\prod_{i=1}^{i'_2-1}(n_i+k_{i}^{(F)})$}
\end{equation*} 
and
\begin{equation}
\frac{(n_1 + k_1^{(E)}) \cdots (n_{i_2'-1} + k_{i_2'-1}^{(E)})}{(n_1 + k_1^{(F)}) \cdots (n_{i_2'-1} + k_{i_2'-1}^{(F)})} \cdot \frac{(n_{i_2'} + k_{i_2'}^{(E)})\cdots (n_{i_2 - 1} + k_{i_2-1}^{(E)})}{n_{i_2'}\cdots n_{i_2-1}} > 1.
\end{equation}
In the same way as above, one obtains a unital homomorphism $$\tilde{\phi}_{i'_2, i_2}^{(F, E)}: \mathrm{M}_{m_{i'_2}^{(F)}}(\mathrm{C}(X^{d_{i'_2}})) \to \mathrm{M}_{m_{i_2}^{(E)}}(\mathrm{C}(X^{d_{i_2}})) $$
$$f \mapsto \mathrm{diag}\{\underbrace{f\circ \pi_1, ..., f\circ\pi_{c_{i_2'} \cdots c_{i_2-1}}}_{n_{i_2'} \cdots n_{i_2-1}}, \textrm{point evaluations}\},$$ such that, with
$$\phi_{i_1, i_2}^{(F, E)}=\tilde{\phi}_{i'_2, i_2}^{(F, E)} \circ \phi_{i_1, i'_2}^{(F)},$$
and compressing, we have the augmented diagram 
$$
\xymatrix{
\mathrm{M}_{n_0}(\mathrm{C}(X)) \ar[r]^-{\phi_{1, i_1}^{(E)}} \ar[dr]^-{\phi_{1, i_1}^{(E, F)}} & \mathrm{M}_{m^{(E)}_{i_1}}(\mathrm{C}(X^{d_{i_1}})) \ar[r]^-{\phi_{i_1, i_2}^{(E)}}  & \mathrm{M}_{m^{(E)}_{i_2}}(\mathrm{C}(X^{d_{i_2}})) \ar[r] & \cdots \ar[r] & A_E \\
\mathrm{M}_{n_0}(\mathrm{C}(X)) \ar[r]^-{\phi_{1, i_1}^{(F)}}  & \mathrm{M}_{m^{(F)}_{i_1}}(\mathrm{C}(X^{d_{i_1}})) \ar[r]^-{\phi_{i_1, i_2}^{(F)}} \ar[ur]^-{\phi_{i_1, i_2}^{(F, E)}}  & \mathrm{M}_{m^{(F)}_{i_2}}(\mathrm{C}(X^{d_{i_2}})) \ar[r] & \cdots \ar[r] & A_F.
}
$$
Note that, by \eqref{small-E-1},
$$\abs{\tau(\phi^{(F, E)}_{i_{1}, i_{2}} \circ \phi^{(E, F)}_{1, i_{1}}(h) - \phi^{(E)}_{i_{1}, i_{2}} \circ \phi^{(E)}_{1, i_1}(h))} < \delta_1,\quad h\in \mathrm{C}(X),$$ and
$\phi^{(F, E)}_{i_{1}, i_{2}} \circ \phi^{(E, F)}_{1, i_{1}}$ is a map  
$$f \mapsto \mathrm{diag}(\underbrace{f\circ \pi_1, ..., f\circ\pi_{c_{i_1} \cdots c_{i_2-1}}}_{n_{i_1} \cdots n_{i_2-1}}, \textrm{point evaluations}).$$

Repeating this process, we have $i_1 < i_2 < \cdots $ with
\begin{equation}\label{small-trace-0}
1 - \prod_{j=0}^\infty \frac{n_{i_s+j}}{n_{i_s+j}+k^{(E)}_{i_s+j}} < \delta_s\quad\mathrm{and}\quad 1 - \prod_{j=0}^\infty \frac{n_{i_s+j}}{n_{i_s+j}+k^{(F)}_{i_s+j}} < \delta_s,\quad s=1, 2, ...,
\end{equation}
and the infinite intertwining diagram
\begin{equation*}
\xymatrix{
\mathrm{M}_{n_0}(\mathrm{C}(X)) \ar[r]^-{\phi_{1, i_1}^{(E)}} \ar[dr]^-{\phi_{1, i_1}^{(E, F)}} & \mathrm{M}_{m^{(E)}_{i_1}}(\mathrm{C}(X^{d_{i_1}})) \ar[r]^-{\phi_{i_1, i_2}^{(E)}}  & \mathrm{M}_{m^{(E)}_{i_2}}(\mathrm{C}(X^{d_{i_2}})) \ar[r] \ar[dr]^-{\phi_{i_2, i_3}^{(E, F)}} & \cdots \ar[r] & A_E \\
\mathrm{M}_{n_0}(\mathrm{C}(X)) \ar[r]^-{\phi_{1, i_1}^{(F)}}  & \mathrm{M}_{m^{(F)}_{i_1}}(\mathrm{C}(X^{d_{i_1}})) \ar[r]^-{\phi_{i_1, i_2}^{(F)}} \ar[ur]^-{\phi_{i_1, i_2}^{(F, E)}}  & \mathrm{M}_{m^{(F)}_{i_2}}(\mathrm{C}(X^{d_{i_2}})) \ar[r] & \cdots \ar[r] & A_F.
}
\end{equation*}

The diagram \eqref{diag-0} is not commutative. But, by \eqref{small-trace-0}, we have
$$\abs{\tau(\phi^{(F, E)}_{i_{s+1}, i_{s+2}} \circ \phi^{(E, F)}_{i_s, i_{s+1}}(h) - \phi^{(E)}_{i_{s+1}, i_{s+2}} \circ \phi^{(E)}_{i_s, i_{s+1}}(h))} < \delta_s$$
for any $s=0, 2, ...$, any $h\in \mathrm{M}_{m^{(E)}_{i_s}}(\mathrm{C}(X^{d_{i_s}}))$ with $\norm{h}\leq 1$, and any $\tau\in\mathrm{T}(\mathrm{M}_{m^{(E)}_{i_{s+2}}}(\mathrm{C}(X^{d_{i_{s+2}}})))$; and
$$\abs{\tau(\phi^{(E, F)}_{i_{s+1}, i_{s+2}} \circ \phi^{(F, E)}_{i_s, i_{s+1}}(h) - \phi^{(F)}_{i_{s+1}, i_{s+2}} \circ \phi^{(F)}_{i_s, i_{s+1}}(h))} < \delta_s$$
for any $s=1, 3, ...$, any $h\in \mathrm{M}_{m^{(F)}_{i_s}}(\mathrm{C}(X^{d_{i_s}}))$ with $\norm{h} \leq 1$, and any $\tau\in\mathrm{T}(\mathrm{M}_{m^{(F)}_{i_{s+2}}}(\mathrm{C}(X^{d_{i_{s+2}}})))$. That is, the diagram \eqref{diag-0} is approximately commutative at the level of traces. Note that this implies that the simplices $\mathrm{T}(A_E)$ and $\mathrm{T}(A_F)$ are isomorphic. Moreover, the maps $\phi^{(F, E)}_{i_{s+1}, i_{s+2}} \circ \phi^{(E, F)}_{i_s, i_{s+1}}$ and $\phi^{(E)}_{i_{s+1}, i_{s+2}} \circ \phi^{(E)}_{i_s, i_{s+1}}$ share the same coordinate projection part, and so also do the maps $\phi^{(E, F)}_{i_{s+1}, i_{s+2}} \circ \phi^{(F, E)}_{i_s, i_{s+1}}$ and $\phi^{(F)}_{i_{s+1}, i_{s+2}} \circ \phi^{(F)}_{i_s, i_{s+1}}$.
\end{proof}

\begin{rem}
A direct consequence of \eqref{diag-0} is that the trace simplex of $A_E$ is homeomorphic to that of $A_F$.
In the case of a Goodearl algebra (i.e., $c_i=1$, $i=1, 2, ...$), the trace simplex is homeomorphic to the Bauer simplex with extreme boundary $X$, while, as we have shown, in the case of the Villadsen algebra (i.e., $s_{i, j}=1$, $i=1, 2, ...,$ $j=1, ..., c_i$),  the trace simplex is the Poulsen simplex.
\end{rem}

\section{A uniqueness theorem}

\begin{thm}\label{unique}
Let $X$ be a connected metrizable compact space, and let $\Delta: \mathrm{C}^+_1(X) \to (0, 1]$ be an order preserving map. Then, for any finite set $\mathcal F\subseteq \mathrm{C}(X)$ and any $\eps>0$, there exist finite sets $\mathcal H_0, \mathcal H_1\subseteq \mathrm{C}^+(X)$ and $\delta>0$ such that for any unital homomorphisms $\phi_0, \phi_1: \mathrm{C}(X) \to \mathrm{M}_{n+k}(\mathrm{C}(X^d))$ with
$$\phi_0(f) = \mathrm{diag}\{f\circ\pi_1, ..., f\circ\pi_n, f(x_{1}), ..., f(x_{k})\}$$
and
$$\phi_1(f) = \mathrm{diag}\{f\circ\pi_1, ..., f\circ\pi_n, f(y_{1}), ..., f(y_{k})\},$$
where $x_1, ..., x_k$ and $y_1, ..., y_k$ are points of $X$ and $\pi_1, ..., \pi_n$ are coordinate projections (possibly with multiplicity), if 
$$\tau(\phi_0(h)), \tau(\phi_1(h)) > \Delta(h),\quad h \in \mathcal H_0, $$
and
$$\abs{\tau( \phi_0(h) - \phi_1(h))} < \delta,\quad h\in\mathcal H_1,\ \tau\in\mathrm{T}(\mathrm{M}_{n+k}(\mathrm{C}(X^d))),$$
then there is a unitary $u \in \mathrm{M}_{n+k}(\mathrm{C}(X^d))$ such that 
$$\norm{\phi_0(f) - u^*\phi_1(f)u} < \eps,\quad f\in\mathcal{F}.$$
\end{thm}

\begin{proof}
Fix a metric for $X$. Since $X$ is compact, there is $\eta>0$ such that for any $x, y \in X$ with $
\mathrm{dist}(x, y) < 3\eta$, one has $$\abs{f(x) - f(y)} < \eps,\quad f\in\mathcal F.$$

Choose an open cover $$\mathcal U=\{U_1, U_2, ..., U_{\abs{\mathcal U}}\}$$ with each $U_i$ of diameter at most $\eta$. Let $$\mathcal O = \{O_1, O_2, ..., O_S\}$$ denote the set of all finite unions of the sets $U_1, U_2, ..., U_{\abs{\mathcal U}}$. For each $O\in\mathcal O$, define
$$h_O(x) = \max\{1 - \mathrm{dist}(x, O)/\eta, 0\},\quad x\in X.$$ Also, for each $O\in\mathcal O$ with $O_\eta\neq X$, where $O_\eta$ denotes the $\eta$-neighborhood of $O$ (hence $O_{2\eta}\setminus O_\eta \neq \O$, as otherwise $O_\eta$ is a clopen set and $X$ is assumed to be connected), choose a non-zero positive function $g_O\in\mathrm{C}(X)$ such that $g_O \leq 1$ and
$$\mathrm{supp}(g_O) \subseteq O_{2\eta}\setminus O_\eta.$$
Then
$$\mathcal H_0 := \{g_O: O\in\mathcal O,\ O_\eta\neq X\},\quad \mathcal H_1 := \{h_O: O\in\mathcal O\},\quad\mathrm{and}\quad \delta:=\min\{\Delta(g_O): O\in\mathcal O\}$$
possess the property of the lemma.

Let $\phi_0$ and $\phi_1$ be given as in the statement of the lemma. Let $\tilde{X} \subseteq \{x_1, x_2, ..., x_k\}$ be an arbitrary subset. Let $U_{i_1}, U_{i_2}, ..., U_{i_l}\in \mathcal U$ be such that $U_{i_j} \cap \tilde{X} \neq \O$, and consider the union
$$O= U_{i_1} \cup \cdots \cup U_{i_l} \in \mathcal O.$$ 

Assume $O_{2\eta} \neq X$ (so that $O_{\eta}\neq X$), and choose 
$$x'_O \in X\setminus O_{2\eta} ,$$
and then choose $x_O \in X^d$ (e.g., pick $x_O=(x'_O, ..., x'_O)$) such that
$$\pi_1(x_O), ... , \pi_n(x_O) \in X\setminus O_{2\eta}.$$
Then
\begin{eqnarray*}
\abs{\tilde{X}} & \leq & (n+k)\mathrm{tr}_{x_O}(\phi_0(h_O)) \\
& \leq & (n+k)\mathrm{tr}_{x_O}(\phi_1(h_O)) + (n+k)\delta \\
& \leq & \abs{O_\eta \cap \{y_1, y_2, ..., y_k\}} + (n+k)\delta \\
& \leq & \abs{O_\eta \cap \{y_1, y_2, ..., y_k\}} + (n+k)\Delta(g_O)  \\
& \leq & \abs{O_\eta \cap \{y_1, y_2, ..., y_k\}} + (n+k)\mathrm{tr}_{x_O}(g_O)  \\
& \leq  & \abs{O_\eta \cap \{y_1, y_2, ..., y_k\}} + \abs{O_{\eta, 2\eta} \cap \{y_1, y_2, ..., y_k\}}  \\
& \leq  & \abs{O_{2\eta} \cap \{y_1, y_2, ..., y_k\}} \\
& \leq &  \abs{\tilde{X}_{3\eta} \cap \{y_1, y_2, ..., y_k\}} \quad\quad \textrm{($O_{2\eta} \subseteq \tilde{X}_{3\eta}$)}.
\end{eqnarray*}

If $O_{2\eta} = X$, then $\tilde{X}_{3\eta} = X$. In particular, we still have
$$\abs{\tilde{X}} \leq k = \abs{\tilde{X}_{3\eta} \cap \{y_1, y_2, ..., y_k\}}.$$ That is, we always have
\begin{equation}
\abs{\tilde{X}} \leq \abs{\tilde{X}_{3\eta} \cap \{y_1, y_2, ..., y_k\}}.
\end{equation}
The same calculation shows that, for any subset $\tilde{Y} \subseteq \{y_1, y_2, ..., y_k\}$, 
\begin{equation}
\abs{\tilde{Y}} \leq \abs{\tilde{Y}_{3\eta} \cap \{x_1, x_2, ..., x_k\}}.
\end{equation}
Thus, by the Marriage Lemma (\cite{Hall-Marriage-Lemma}), there is a one-to-one correspondence
$$\sigma: \{x_1, x_2, ..., x_k\} \to \{y_1, y_2, ..., y_k\}$$
such that
$$\mathrm{dist}(x_i, \sigma{(x_i)}) < 3\eta,\quad i=1, 2, ..., k.$$
Denote by $w\in\mathrm{M}_k(\Comp)$ the permutation unitary that induces $\sigma$. Then $$u = \mathrm{diag}\{1_n, w\}$$ is the desired unitary.
\end{proof}

\section{An isomorphism theorem}

\begin{thm}\label{thm-p}
Assume $X$ is a solid connected metrizable compact space which is finite dimensional. 
Let $A_E$ and $A_E$ be two Villadsen algebras with point-evaluation set $E$ and $F$ respectively (possibly with different numbers of points, but with the same space $X$ and the same $(c_i)$ and $(s_{i, 1}, ..., s_{i, c_i})$).  
Then $A_E \cong A_F$ if, and only if, $$ \rho(\Kzero(A_E)) = \rho(\Kzero(A_F)) \quad \textrm{and}\quad \mathrm{rc}(A_E) = \mathrm{rc}(A_F),$$
where $\rho_{A_E}$ and $\rho_{A_F}$ are the unique states of the order-unit groups $\Kzero(A_E)$ and $\Kzero(A_F)$, respectively. (The solidness condition and finite-dimensionality condition on $X$ are not necessary when the radius of comparison is $0$.)
\end{thm}

\begin{proof}
If $\mathrm{rc}(A_E) = \mathrm{rc}(A_F) = 0$, then $A_E$ and $A_F$ are $\mathcal Z$-stable. Since  $A_E$ and $A_F$ are built with the same $(c_i)$ and $(s_{i, 1}, ..., s_{i, c_i})$, it follows that $\Kone(A_E) \cong \Kone(A_F)$, and by \eqref{rdg-cond} and Lemma \ref{trace-int-1} that $\mathrm{T}(A_E) \cong\mathrm{T}(A_F)$. For each $d\in\mathbb N$, write $\Kzero(\mathrm{C}(X^d)) = \Int\oplus H_d$, where $H_d$ consists of the $\Kzero$-elements vanishing on traces of $\mathrm{C}(X^d)$, and note that the image of $H_{c_1\cdots c_i}$ under the connecting map is inside  $H_{c_1\cdots c_{i+1}}$  and is independent of the choices of the point evaluations. Then, denoting by $H$ the limit of $(H_{c_1\cdots c_i})$, which depends only on $(c_i)$ and $(s_{i, 1}, ..., s_{i, c_i})$, one has $$\Kzero(A_E) \cong \rho(\Kzero(A_E))\oplus H \quad \mathrm{and} \quad \Kzero(A_E) \cong \rho(\Kzero(A_E))\oplus H $$ as abelian groups. Since $\rho(\Kzero(A_E)) = \rho(\Kzero(A_F))$ and $A_E$ and $A_F$ are $\mathcal Z$-stable (so that the strict order on the $\Kzero$-group is determined by the traces), one has $\Kzero(A_E) \cong \Kzero(A_F)$ as order-unit groups. 
Therefore, $$((\Kzero(A_E), \Kzero^+(A_E), [1]_0), \Kone(A_E), \mathrm{T}(A), \rho_A) \cong  ((\Kzero(A_F), \Kzero^+(A_F), [1]_0), \Kone(A_F),  \mathrm{T}(B), \rho_B),$$ and hence $A_E\cong A_F$ (see \cite{ENST-ASH} and \cite{EGLN-ASH}). 

Now, assume $\mathrm{rc}(A_E) = \mathrm{rc}(A_F) \neq 0$. Since $X$ is solid, by Theorem \ref{rcA}, we have 
$$\frac{\mathrm{dim}(X)}{n_0} \cdot \lim_{i\to\infty}\frac{c_1\cdots c_i}{(n_1+k^{(E)}_1)\cdots (n_i+k^{(E)}_i)} = \frac{\mathrm{dim}(X)}{n_0} \cdot \lim_{i\to\infty} \frac{c_1\cdots c_i}{(n_1+k^{(F)}_1)\cdots (n_i+k^{(F)}_i)}.$$
Since $\mathrm{dim}(X)<\infty$, both sides are finite non-zero numbers, and 
$$\lim_{i\to\infty}\frac{c_1\cdots c_i}{(n_1+k^{(E)}_1)\cdots (n_i+k^{(E)}_i)} = \lim_{i\to\infty} \frac{c_1\cdots c_i}{(n_1+k^{(F)}_1)\cdots (n_i+k^{(F)}_i)}.$$
Since the limits are not $0$ (otherwise, the radius of comparison is $0$), the ratio of the two sequences above converges to $1$, i.e.,
$$ \lim_{i\to\infty} \frac{(n_1+k^{(E)}_1)\cdots (n_i+k^{(E)}_i)}{(n_1+k^{(F)}_1)\cdots (n_i+k^{(F)}_i) } = 1.$$

Consider the inductive limit constructions
$$
\xymatrix{
\mathrm{M}_{n_0}(\mathrm{C}(X) \ar[r]^-{\phi_{1}^{(E)}}) & \mathrm{M}_{n_0(n_1+k^{(E)}_1)}(\mathrm{C}(X^{c_1})) \ar[r]^-{\phi_{2}^{(E)}}  & \mathrm{M}_{n_0(n_1+k^{(E)}_1)(n_2+k^{(E)}_2)}(\mathrm{C}(X^{c_1c_2})) \ar[r] & \cdots \ar[r] & A_E, \\
\mathrm{M}_{n_0}(\mathrm{C}(X) \ar[r]^-{\phi_{1}^{(F)}})  & \mathrm{M}_{n_0(n_1+k^{(F)}_1)}(\mathrm{C}(X^{c_1})) \ar[r]^-{\phi_{2}^{(F)}}  & \mathrm{M}_{n_0(n_1+k^{(F)}_1)(n_2+k^{(F)}_2)}(\mathrm{C}(X^{c_1c_2})) \ar[r] & \cdots \ar[r] & A_F.
}
$$

Choose finite subsets $$\mathcal F^{(E)}_1\subseteq \mathrm{M}_{n_0}(C(X)),\ \mathcal F^{(E)}_2\subseteq \mathrm{M}_{n_0(n_1+k^{(E)}_1)}(\mathrm{C}(X^{c_1})), ...$$ 
and 
$$\mathcal F^{(F)}_1\subseteq \mathrm{M}_{n_0}(C(X)),\  \mathcal F^{(F)}_2\subseteq \mathrm{M}_{n_0(n_1+k^{(F)}_1)}(\mathrm{C}(X^{c_1})), ... $$ such that
$$\overline{\bigcup_{i=1}^\infty \mathcal F_i^{(E)}} = A_E \quad \mathrm{and}\quad \overline{\bigcup_{i=1}^\infty \mathcal F_i^{(F)}} = A_F.$$
Also, choose
$\eps_1 > \eps_2 > \cdots > 0$
such that
$$\sum_{i=1}^\infty \eps_i \leq 1.$$

Since $A_E$ and $A_F$ are simple, we have 
$$\Delta_E(h) := \inf\{\tau(h): \tau\in\mathrm{T}(A_E)\}>0,\quad h\in A_E^+\setminus\{0\},$$ 
and
$$ \Delta_F(h) := \inf\{\tau(h): \tau\in\mathrm{T}(A_F)\}>0,\quad h\in A^+_F\setminus\{0\}.$$

Applying Theorem \ref{unique} to $(\mathcal F_i^{(E)}, \eps_i)$, we obtain finite sets $\mathcal H_{i, 0}^{(E)}, \mathcal H_{i, 1}^{(E)} \subseteq \mathrm{M}_{m^{(E)}_i}(\mathrm{C}(X^{d_i}))$ and $\delta^{(E)}_i>0$. 
Applying Theorem \ref{unique} to $(\mathcal F_i^{(F)}, \eps_i)$, we obtain finite sets $\mathcal H_{i, 0}^{(F)}, \mathcal H_{i, 1}^{(F)} \subseteq \mathrm{M}_{m^{(F)}_i}(\mathrm{C}(X^{d_i}))$ and $\delta^{(F)}_i>0$. Set $\delta_i=\min\{\delta_i^{(E)}, \delta_i^{(F)}\}$.

Applying Lemma \ref{trace-int}, we have the diagram
\begin{equation}\label{diag-1}
\xymatrix{
\mathrm{M}_{n_0}(\mathrm{C}(X)) \ar[r]^-{\phi_{1, i_1}^{(E)}} \ar[dr]^-{\phi_{1, i_1}^{(E, F)}} & \mathrm{M}_{m^{(E)}_{i_1}}(\mathrm{C}(X^{d_{i_1}})) \ar[r]^-{\phi_{i_1, i_2}^{(E)}}  & \mathrm{M}_{m^{(E)}_{i_2}}(\mathrm{C}(X^{d_{i_2}})) \ar[r] \ar[dr]^-{\phi_{i_2, i_3}^{(E, F)}} & \cdots \ar[r] & A_E \\
\mathrm{M}_{n_0}(\mathrm{C}(X)) \ar[r]^-{\phi_{1, i_1}^{(F)}}  & \mathrm{M}_{m^{(F)}_{i_1}}(\mathrm{C}(X^{d_{i_1}})) \ar[r]^-{\phi_{i_1, i_2}^{(F)}} \ar[ur]^-{\phi_{i_1, i_2}^{(F, E)}}  & \mathrm{M}_{m^{(F)}_{i_2}}(\mathrm{C}(X^{d_{i_2}})) \ar[r] & \cdots \ar[r] & A_F
}
\end{equation}
which is approximately commutative at the level of traces: that is,
$$\abs{\tau(\phi^{(F, E)}_{i_{s+1}, i_{s+2}} \circ \phi^{(E, F)}_{i_s, i_{s+1}}(h) - \phi^{(E)}_{i_{s+1}, i_{s+2}} \circ \phi^{(E)}_{i_s, i_{s+1}}(h))} < \delta_s,$$
for any $s=0, 2, ...$, any $h\in \mathrm{M}_{m^{(E)}_{i_s}}(\mathrm{C}(X^{d_{i_s}}))$ with $\norm{h} \leq 1$, and any $\tau\in\mathrm{T}(\mathrm{M}_{m^{(E)}_{i_{s+2}}}(\mathrm{C}(X^{d_{i_{s+2}}})))$; and, furthermore, 
$$\abs{\tau(\phi^{(E, F)}_{i_{s+1}, i_{s+2}} \circ \phi^{(F, E)}_{i_s, i_{s+1}}(h) - \phi^{(F)}_{i_{s+1}, i_{s+2}} \circ \phi^{(F)}_{i_s, i_{s+1}}(h))} < \delta_s,$$
for any $s=1, 3, ...$, any $h\in \mathrm{M}_{m^{(F)}_{i_s}}(\mathrm{C}(X^{d_{i_s}}))$ with $\norm{h} \leq 1$, and any $\tau\in\mathrm{T}(\mathrm{M}_{m^{(F)}_{i_{s+2}}}(\mathrm{C}(X^{d_{i_{s+2}}})))$.
Moreover, the maps $\phi^{(F, E)}_{i_{s+1}, i_{s+2}} \circ \phi^{(E, F)}_{i_s, i_{s+1}}$ and $\phi^{(E)}_{i_{s+1}, i_{s+2}} \circ \phi^{(E)}_{i_s, i_{s+1}}$ share the same coordinate projection part, and so also do the maps $\phi^{(E, F)}_{i_{s+1}, i_{s+2}} \circ \phi^{(F, E)}_{i_s, i_{s+1}}$ and $\phi^{(F)}_{i_{s+1}, i_{s+2}} \circ \phi^{(F)}_{i_s, i_{s+1}}$.

Then, by Theorem \ref{unique}, there are unitaries
$$u^{(E)}_2\in \mathrm{M}_{m_{i_2}}(\mathrm{C}(X^{d_{i_2}})),\quad u^{(E)}_4\in \mathrm{M}_{m_{i_4}}(\mathrm{C}(X^{d_{i_4}})),\ ...$$
and
$$u^{(E)}_3\in \mathrm{M}_{m_{i_3}}(\mathrm{C}(X^{d_{i_3}})),\quad u^{(E)}_5\in \mathrm{M}_{m_{i_5}}(\mathrm{C}(X^{d_{i_5}})),\ ...$$
such that 
$$\norm{\phi^{(F, E)}_{i_{s+1}, i_{s+2}} \circ \phi^{(E, F)}_{i_s, i_{s+1}}(f) - (u_{s+2}^{(E)})^*(\phi^{(E)}_{i_{s+1}, i_{s+2}} \circ \phi^{(E)}_{i_s, i_{s+1}}(f))u_{s+2}^{(E)}} < \eps_s$$
for any $s=0, 2, 4...$, any $f\in \mathcal F_{i_s}^{(E)} \subseteq \mathrm{M}_{m^{(E)}_{i_s}}(\mathrm{C}(X^{d_{i_s}}))$; and, furthermore, 
$$\norm{\phi^{(E, F)}_{i_{s+1}, i_{s+2}} \circ \phi^{(F, E)}_{i_s, i_{s+1}}(f) - (u_{s+2}^{(F)})^*\phi^{(F)}_{i_{s+1}, i_{s+2}} \circ \phi^{(F)}_{i_s, i_{s+1}}(f)u_{s+2}^{(F)}} < \eps_s$$
for any $s=1, 3, ...$, any $f\in \mathcal F_{i_s}^{(F)}  \subseteq \mathrm{M}_{m^{(F)}_{i_s}}(\mathrm{C}(X^{d_{i_s}}))$.

That is, the $i$th triangle of the diagram \begin{equation}\label{diag-2}
\xymatrix{
\mathrm{M}_{n_0}(\mathrm{C}(X)) \ar[r]^-{\phi_{1, i_1}^{(E)}} \ar[dr]^-{\phi_{1, i_1}^{(E, F)}} & \mathrm{M}_{m^{(E)}_{i_1}}(\mathrm{C}(X^{d_{i_1}})) \ar[r]^-{\mathrm{ad}(u_2^{(E)})\circ\phi_{i_1, i_2}^{(E)}}  & \mathrm{M}_{m^{(E)}_{i_2}}(\mathrm{C}(X^{d_{i_2}})) \ar[r] \ar[dr]^-{\phi_{i_2, i_3}^{(E, F)}} & \cdots \ar[r] & A_E \\
\mathrm{M}_{n_0}(\mathrm{C}(X)) \ar[r]_-{\phi_{1, i_1}^{(F)}}  & \mathrm{M}_{m^{(F)}_{i_1}}(\mathrm{C}(X^{d_{i_1}})) \ar[r]_-{\phi_{i_1, i_2}^{(F)}} \ar[ur]^-{\phi_{i_1, i_2}^{(F, E)}}  & \mathrm{M}_{m^{(F)}_{i_2}}(\mathrm{C}(X^{d_{i_2}})) \ar[r]_-{ \mathrm{ad}(u_3^{(F)})\circ \phi_{i_1, i_2}^{(F)}} & \cdots \ar[r] & A_F
}
\end{equation}
is approximately commutative (pointwise) in norm, to within tolerance $(\mathcal F_s^{(E)}, \eps_s)$ or $(\mathcal F_s^{(F)}, \eps_s)$. Then, by the approximate intertwining argument (Theorems 2.1 and 2.2 of \cite{Ell-AT-RR0}), we have $$A_E \cong A_F,$$ as desired.
\end{proof}

\begin{rem}
In the case of Villadsen algebras, i.e., with coordinate projections of multiplicity one (and rapid dimension growth), both the trace simplex of $A(X, (n_i), (k_i))$ and the trace simplex of $A(X^2, (n_i), (k_i))$ are isomorphic to the Poulsen simplex (Theorem \ref{trace-VA}). However, the algebras $A(X, (n_i), (k_i))$ and $A(X^2, (n_i), (k_i))$ are not isomorphic in general as their radii of comparison (see \cite{RC-Toms}) might be different. In the case that $X$ is contractible, we will show below (Corollary \ref{diff-prod}) that this class of C*-algebras is in fact classified by the order-unit $\Kzero$-group and the radius of comparison.
\end{rem}

\begin{rem}
In the case of a Goodearl algebra (\cite{Goodearl-AH}), i.e., with only one coordinate projection, $\mathcal Z$-stability always holds, as the mean dimension in the
sense of \cite{Niu-MD} is always zero (Theorem \ref{rcA}).
\end{rem}

\section{Let $(n_i)$ vary}

We shall show the following theorem in this section:

\begin{thm}\label{n-theorm}
Let $X$ be a K-contractible (i.e., $\Kzero(\mathrm{C}(X)) = \Int$ and $\Kone(\mathrm{C}(X)) = \{0\}$) solid metrizable compact space which is finite-dimensional. Let 
$$A:=A(X, (n^{(A)}_i), (k^{(A)}_i), E^{(A)}) \quad\textrm{and}\quad B:=B(X, (n^{(B)}_i), (k^{(B)}_i), F^{(B)})$$
be Villadsen algebras (with coordinate projections of arbitrary multiplicity).
Then 
$A \cong B$ 
if, and only if,  $$\Kzero(A) \cong \Kzero(B),\quad \mathrm{T}(A) \cong \mathrm{T}(B), \quad \textrm{and}\quad \mathrm{rc}(A) = \mathrm{rc}(B).$$
Moreover, if $\mathrm{rc}(A) \neq 0$ (or $\mathrm{rc}(B) \neq 0$), then $\mathrm{T}(A)$ (or $\mathrm{T}(B)$) is redundant in the invariant, that is, $A \cong B$ 
if, and only if,  $$\Kzero(A) \cong \Kzero(B) \quad \textrm{and}\quad \mathrm{rc}(A) = \mathrm{rc}(B).$$
\end{thm}

\begin{rem}
Since $X$ is assumed to be K-contractible, we have
$$\Kzero(A)\cong  \Int[\frac{1}{n_0^{(A)}}, \frac{1}{n_1^{(A)} + k_1^{(A)}}, ...] \subseteq \Ratn$$
and
$$\Kzero(B)\cong  \Int[\frac{1}{n_0^{(B)}}, \frac{1}{n_1^{(B)} + k_1^{(B)}}, ...] \subseteq \Ratn,$$
with the class of the unit being of course $1\in\mathbb Z$.
\end{rem}

\begin{rem}
All contractible spaces are K-contractible, but not all K-contractible spaces are contractible. Instances of this are the 2-skeleton of the Poincar\'{e} homology 3-sphere (or the Poincar\'{e} homology 3-sphere with a small open ball removed), and the join of two infinite brooms:

\begin{center}
      \includegraphics[scale=0.6]{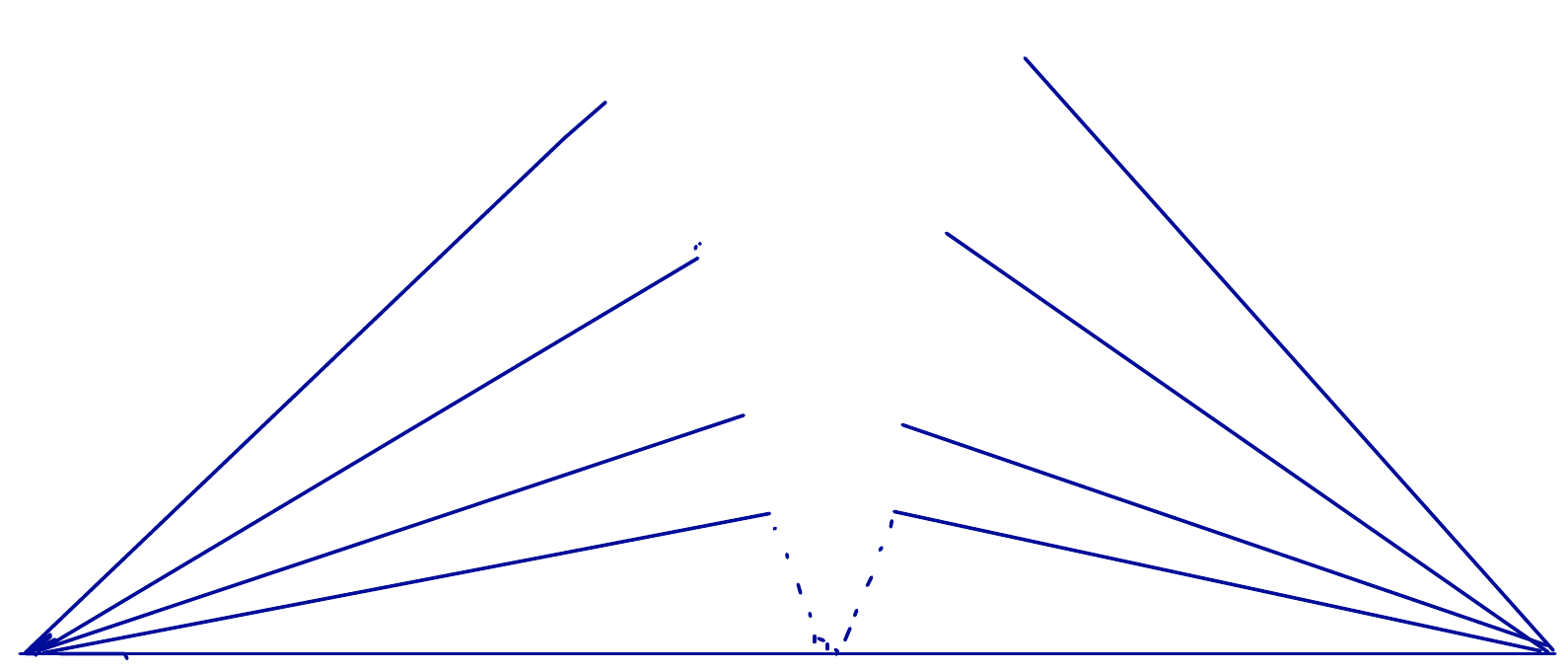}
   \end{center}

\end{rem}

\subsection{An intertwining diagram}
\begin{lem}\label{trace-twining-2}
With $X$ a metrizable compact space, let 
$$A:=A(X, (n^{(A)}_i), (k^{(A)}_i), E^{(A)}) \quad\textrm{and}\quad B:=B(X, (n^{(B)}_i), (k^{(B)}_i), F^{(B)})$$
be Villadsen algebras.
Assume that
\begin{equation}\label{super-n}
n_0^{(A)}\prod_{i=1}^\infty(n_i^{(A)} + k_i^{(A)}) = n_0^{(B)}\prod_{i=1}^\infty(n_i^{(B)} + k_i^{(B)}),
\end{equation} 
as supernatural numbers, and
\begin{equation}\label{same-rc}
\frac{1}{n_0^{(A)}}\prod_{i=1}^\infty \frac{c^{(A)}_i}{n^{(A)}_i + k^{(A)}_i} = \frac{1}{n_0^{(B)}} \prod_{i=1}^\infty \frac{c^{(B)}_i}{n^{(B)}_i + k^{(B)}_i} \neq 0,
\end{equation} 
as real numbers.
Let $\delta_1, \delta_2, ...$ be a decreasing sequence of strictly positive numbers with $$ \sum_{n=1}^\infty \delta_n < 1.$$ 
Then there is a diagram
\begin{equation}\label{diag-0-7}
\xymatrix{
\mathrm{M}_{n_0^{(A)}}(\mathrm{C}(X)) \ar[r]^-{\phi_{1, i_1}^{(A)}} \ar[dr]^{\phi_{1, i_1}^{(A, B)}} & \mathrm{M}_{m^{(A)}_{i_1}}(\mathrm{C}(X^{d^{(A)}_{i_1}})) \ar[r]^-{\phi_{i_1, i_2}^{(A)}}  & \mathrm{M}_{m^{(A)}_{i_2}}(\mathrm{C}(X^{d^{(A)}_{i_2}})) \ar[r] \ar[dr]^{\phi_{i_2, i_3}^{(A, B)}} & \cdots \ar[r] & A \\
\mathrm{M}_{n_0^{(B)}}(\mathrm{C}(X)) \ar[r]^-{\phi_{1, i_1}^{(B)}}  & \mathrm{M}_{m^{(B)}_{i_1}}(\mathrm{C}(X^{d^{(B)}_{i_1}})) \ar[r]^-{\phi_{i_1, i_2}^{(B)}} \ar[ur]^{\phi_{i_1, i_2}^{(B, A)}}  & \mathrm{M}_{m^{(F)}_{i_2}}(\mathrm{C}(X^{d^{(B)}_{i_2}})) \ar[r] & \cdots \ar[r] & B,
}
\end{equation}
where
$$m_i:=n_0(n_1+k_1)\cdots (n_{i-1}+k_{i-1}), \quad d_i:=c_1\cdots c_{i-1},$$
such that
$$\abs{\tau(\phi^{(B, A)}_{i_{s+1}, i_{s+2}} \circ \phi^{(A, B)}_{i_s, i_{s+1}}(h) - \phi^{(A)}_{i_{s+1}, i_{s+2}} \circ \phi^{(A)}_{i_s, i_{s+1}}(h))} < \delta_{i_s}$$
for any $s=0, 2, ...$, any $h\in \mathrm{M}_{m^{(A)}_{i_s}}(\mathrm{C}(X^{d^{(A)}_{i_s}}))$ with $\norm{h}\leq 1$, and any $\tau\in\mathrm{T}(\mathrm{M}_{m^{(A)}_{i_{s+2}}}(\mathrm{C}(X^{d^{(A)}_{i_{s+2}}})))$; 
and, furthermore,
$$\abs{\tau(\phi^{(A, B)}_{i_{s+1}, i_{s+2}} \circ \phi^{(B, A)}_{i_s, i_{s+1}}(h) - \phi^{(B)}_{i_{s+1}, i_{s+2}} \circ \phi^{(B)}_{i_s, i_{s+1}}(h))} < \delta_{i_s}$$
for any $s=1, 3, ...$, any $h\in \mathrm{M}_{m^{(B)}_{i_s}}(\mathrm{C}(X^{d^{(B)}_{i_s}}))$ with $\norm{h}\leq 1$, and any $\tau\in\mathrm{T}(\mathrm{M}_{m^{(B)}_{i_{s+2}}}(\mathrm{C}(X^{d^{(B)}_{i_{s+2}}})))$.

Moreover, for each $s = 0, 2, ... $,
$$\phi^{(B, A)}_{i_{s+1}, i_{s+2}} \circ \phi^{(A, B)}_{i_s, i_{s+1}} =\mathrm{diag}\{P_s, R'_s, \Theta'_s \} $$ 
and
$$\phi^{(A)}_{i_{s}, i_{s+2}} = \phi^{(A)}_{i_{s+1}, i_{s+2}} \circ \phi^{(A)}_{i_s, i_{s+1}} = \mathrm{diag}\{P_s, R''_s,  \Theta_s''\}, $$
where $P_s$ is a coordinate projection, and $\Theta'_s$ and $\Theta''_s$ are point evaluations with
$$ \mathrm{rank}(\Theta_s') = \mathrm{rank}(\Theta_s'')\quad\textrm{and}\quad  \frac{ \mathrm{rank}(R'_s)} {\mathrm{rank}(\Theta'_s)}  = \frac{ \mathrm{rank}(R''_s)} {\mathrm{rank}(\Theta''_s)} < \delta_{i_s},$$
and
for each $s = 1, 3, ...$,
$$\phi^{(A, B)}_{i_{s+1}, i_{s+2}} \circ \phi^{(B, A)}_{i_s, i_{s+1}} =\mathrm{diag}\{P_s, R'_s, \Theta'_s \} $$ 
and
$$\phi^{(B)}_{i_{s}, i_{s+2}} = \phi^{(B)}_{i_{s+1}, i_{s+2}} \circ \phi^{(A)}_{i_s, i_{s+1}} = \mathrm{diag}\{P_s, R''_s,  \Theta_s''\}, $$
where $P_s$ is a coordinate projection, and $\Theta'_s$ and $\Theta''_s$ are point evaluations with
$$ \mathrm{rank}(\Theta_s') = \mathrm{rank}(\Theta_s'')\quad\textrm{and}\quad  \frac{ \mathrm{rank}(R'_s)} {\mathrm{rank}(\Theta'_s)}  = \frac{ \mathrm{rank}(R''_s)} {\mathrm{rank}(\Theta''_s)} < \delta_{i_s}.$$
\end{lem}

\begin{proof}
Consider the inductive constructions
$$
\xymatrix{
\mathrm{M}_{n_0^{(A)}}(\mathrm{C}(X)) \ar[r]^-{\phi_{1}^{(A)}} & \mathrm{M}_{m_2^{(A)}}(\mathrm{C}(X^{d^{(A)}_2})) \ar[r]^-{\phi_{2}^{(A)}}  & \mathrm{M}_{m_3^{(A)}}(\mathrm{C}(X^{d^{(A)}_3})) \ar[r] & \cdots \ar[r] & A, \\
\mathrm{M}_{n_0^{(B)}}(\mathrm{C}(X)) \ar[r]^-{\phi_{1}^{(B)}}  & \mathrm{M}_{m_2^{(B)}}(\mathrm{C}(X^{d^{(B)}_2})) \ar[r]^-{\phi_{2}^{(B)}}  & \mathrm{M}_{m_3^{(B)}}(\mathrm{C}(X^{d^{(B)}_3})) \ar[r] & \cdots \ar[r] & B,
}
$$
where
$$m_i:=n_0(n_1+k_1)\cdots (n_{i-1}+k_{i-1}), \quad d_i:=c_1\cdots c_{i-1}.$$

Set (see \eqref{same-rc})
$$ \gamma:= \lim_{i\to\infty}\frac{c^{(A)}_1 \cdots c^{(A)}_i}{n^{(A)}_0(n^{(A)}_1+k^{(A)}_1)\cdots (n^{(A)}_i+k^{(A)}_i)} = \lim_{i\to\infty}\frac{c^{(B)}_1 \cdots c^{(B)}_i}{n^{(B)}_0(n^{(B)}_1+k^{(B)}_1)\cdots (n^{(B)}_i+k^{(B)}_i)}  \in (0, 1).$$

Without loss of generality, since $k_i^{(A)} > 0$, $i=1, 2, ...$, we may assume 
\begin{equation}\label{delta-assumption}
\delta_1 < \frac{k^{(A)}_1}{n^{(A)}_1 + k^{(A)}_1}\quad\mathrm{and}\quad \frac{\frac{3}{4}\delta_1}{1-\frac{3}{4}\delta_1} < \delta_1 <1.
\end{equation} 
There is $i'_1>0$  such that
\begin{equation}\label{small-E-1-n}
1 - \prod_{j=0}^\infty \frac{n^{(A)}_{i'_1+j}}{n^{(A)}_{i'_1+j}+k^{(A)}_{i'_1+j}} < \delta_1,
\end{equation}
and, by Theorem \ref{cor-almost-1}, $i_1'$ can be chosen sufficiently large that for all $j=1, 2, ...$, the ratio $\frac{c_{i'_1}^{(A)} \cdots c^{(A)}_{i'_1+j}}{n_{i'_1}^{(A)} \cdots n^{(A)}_{i'_1+j}}$ is sufficiently close to $1$ that
\begin{equation}\label{small-c-c-ratio} 
 \frac{c_{i'_1}^{(A)} \cdots c^{(A)}_{i'_1+j}}{n_{i'_1}^{(A)} \cdots n^{(A)}_{i'_1+j}}( ( \frac{n_{i'_1}^{(A)} \cdots n^{(A)}_{i'_1+j}}{c_{i'_1}^{(A)} \cdots c^{(A)}_{i'_1+j}} - 1) +\frac{\delta_1^2}{6}) < \frac{\delta_1^2}{3}
\end{equation}
and
\begin{equation}\label{almost-m-1-A}
 \frac{\abs{\{s_k: s_k=1,: k=1, ..., c^{(A)}_{i'_1}\cdots c^{(A)}_{i'_1+j}\}}}{n^{(A)}_{i'_1}\cdots n^{(A)}_{i'_1+j}} > 1-\frac{\delta_1^2}{12},
\end{equation}
where
$$\phi^{(A)}_{i'_1, i'_1+j} = \mathrm{diag}\{ \underbrace{\underbrace{\pi_1^*, ..., \pi_1^*}_{s_1}, ..., \underbrace{\pi_{c^{(A)}_{i_1'}\cdots c^{(A)}_{i_1'+j}}^*, ..., \pi_{c^{(A)}_{i_1'}\cdots c^{(A)}_{i_1'+j}}^*}_{s_{c^{(A)}_{i_1'}\cdots c^{(A)}_{i_1'+j}}}}_{n^{(A)}_{i_1'}\cdots n^{(A)}_{i'_1+j}},\  \mathrm{point\  evaluations}\}.$$

Then, pick $\eps'>0$ such that
\begin{equation}\label{close-to-gamma-0}
 \frac{n^{(A)}_0(n_1^{(A)} + k_1^{(A)}) \cdots (n_{i_1'-1}^{(A)} + k_{i_1'-1}^{(A)}) } {c_1^{(A)}\cdots c_{i'_1-1}^{(A)}} < \frac{1}{\gamma} -\eps',
 \end{equation}
and pick $\eps''>0$ such that
\begin{equation}\label{small-prod}
(\frac{1}{\gamma} - \eps')(\gamma + \eps'') < 1.
\end{equation}

By \eqref{super-n}, there are $i_1>i_1'$ such that
\begin{equation}\label{divide-cond-n}
n^{(B)}_0\textrm{$\prod_{i=1}^{i_1-1}(n^{(B)}_i+k_{i}^{(B)})$ is divisible by $n^{(A)}_0\prod_{i=1}^{i'_1-1}(n^{(A)}_i+k_{i}^{(A)})$},
\end{equation} 
\begin{equation}\label{close-to-gamma-1}
\frac{c_1^{(B)}\cdots c_{i_1}^{(B)}} {n^{(B)}_0(n_1^{(B)} + k_1^{(B)}) \cdots (n_{i_1}^{(B)} + k_{i_1}^{(B)})} < \gamma + \eps''
\end{equation}
and
\begin{equation}\label{small-diff-pre-A-B}
\frac{c_1^{(A)} \cdots c_{i_1'-1}^{(A)}}{c_1^{(B)} \cdots c_{i_1-1}^{(B)}} < \frac{\delta_1^2}{12}.
\end{equation}
By Theorem \ref{cor-almost-1}, one may also assume that for all $j=1, 2, ...$,
\begin{equation}\label{pert-B-C}
\frac{n_{i_1}^{(B)} \cdots n^{(B)}_{i_1+j}}{c_{i_1}^{(B)} \cdots c^{(B)}_{i_1+j}} < 2
\end{equation}
and
\begin{equation}\label{almost-m-1-B}
 \frac{\abs{\{s_k: s_k=1,: k=1, ..., c^{(B)}_{i_1}\cdots c^{(B)}_{i_1+j}\}}}{n^{(B)}_{i_1}\cdots n^{(B)}_{i_1+j}} > 1-\frac{\delta_1^2}{12},
\end{equation}
where
$$\phi^{(B)}_{i_1, i_1+j} = \mathrm{diag}\{ \underbrace{\underbrace{\pi_1^*, ..., \pi_1^*}_{s_1}, ..., \underbrace{\pi_{c^{(B)}_{i_1}\cdots c^{(B)}_{i_1+j}}^*, ..., \pi_{c^{(B)}_{i_1}\cdots c^{(B)}_{i_1+j}}^*}_{s_{c^{(B)}_{i_1}\cdots c^{(B)}_{i_1+j}}}}_{n^{(B)}_{i_1}\cdots n^{(B)}_{i_1+j}},\  \mathrm{point\  evaluations}\}.$$

Pick $l_{i_1', i_1} \in \mathbb N$ such that 
  \begin{equation}\label{approx-div-A-B}
  0\leq c_1^{(B)} \cdots c_{i_1-1}^{(B)} -  (c_1^{(A)} \cdots c_{i_1'-1}^{(A)}) l_{i'_1, i_1} < c_1^{(A)} \cdots c_{i_1'-1}^{(A)}.  
  \end{equation}

Then
\begin{eqnarray}\label{enough-room-n}
&&  \frac{n^{(A)}_0(n_1^{(A)} + k_1^{(A)}) \cdots (n_{i_1'-1}^{(A)} + k_{i_1'-1}^{(A)}) }{n^{(B)}_0(n_1^{(B)} + k_1^{(B)}) \cdots (n_{i_1'-1}^{(B)} + k_{i_1'-1}^{(B)})} \cdot \frac{l_{i_1', i_1}}{(n_{i'_1}^{(B)} + k_{i'_1}^{(B)}) \cdots (n_{i_1}^{(B)} + k_{i_1}^{(B)})}  \\
&\leq &  \frac{n^{(A)}_0 (n_1^{(A)} + k_1^{(A)}) \cdots (n_{i_1'-1}^{(A)} + k_{i_1'-1}^{(A)}) }{n^{(B)}_0(n_1^{(B)} + k_1^{(B)}) \cdots (n_{i_1}^{(B)} + k_{i_1}^{(B)})}   \cdot \frac{c_1^{(B)}\cdots c_{i_1}^{(B)}}{c_1^{(A)}\cdots c_{i'_1-1}^{(A)}} \quad\quad\textrm{(by \eqref{approx-div-A-B})} \nonumber \\
& = & \frac{n^{(A)}_0(n_1^{(A)} + k_1^{(A)}) \cdots (n_{i_1'-1}^{(A)} + k_{i_1'-1}^{(A)}) } {c_1^{(A)}\cdots c_{i'_1-1}^{(A)}}    \cdot \frac{c_1^{(B)}\cdots c_{i_1}^{(B)}} {n^{(B)}_0(n_1^{(B)} + k_1^{(B)}) \cdots (n_{i_1}^{(B)} + k_{i_1}^{(B)})} \nonumber \\
& < & (\frac{1}{\gamma} - \eps')(\gamma +\eps'')<1 \quad\quad\textrm{(by \eqref{close-to-gamma-0}, \eqref{close-to-gamma-1}, and \eqref{small-prod})}. \nonumber
\end{eqnarray}

Then consider the diagram 
$$
\xymatrix{
\mathrm{M}_{n_0^{(A)}}(\mathrm{C}(X)) \ar[r]^-{\phi_{1, i'_1}^{(A)}}  & \mathrm{M}_{m^{(A)}_{i'_1}}(\mathrm{C}(X^{d^{(A)}_{i'_1}})) \ar[r]^-{\phi_{i'_1, i_1}^{(A)}} \ar[dr]^-{\tilde{\phi}_{i'_1, i_1}^{(A, B)}}  & \mathrm{M}_{m^{(A)}_{i_1}}(\mathrm{C}(X^{d^{(A)}_{i_1}})) \ar[r] & \cdots \ar[r] & A \\
\mathrm{M}_{n_0^{(B)}}(\mathrm{C}(X)) \ar[r]^-{\phi_{1, i'_1}^{(B)}}  & \mathrm{M}_{m^{(B)}_{i'_1}}(\mathrm{C}(X^{d^{(B)}_{i'_1}})) \ar[r]^-{\phi_{i'_1, i_1}^{(B)}}  & \mathrm{M}_{m^{(B)}_{i_1}}(\mathrm{C}(X^{d^{(B)}_{i_1}})) \ar[r] & \cdots \ar[r] & B,
}
$$
where 
\begin{equation*}
m_i:=n_0(n_1+k_1)\cdots (n_{i-1}+k_{i-1}), \quad d_i:=c_1\cdots c_{i-1}
\end{equation*}
and
$\tilde{\phi}_{i'_1, i_1}^{(A, B)}: \mathrm{M}_{m^{(A)}_{i'_1}}(\mathrm{C}(X^{d^{(A)}_{i'_1}})) \to \mathrm{M}_{m^{(B)}_{i_1}}(\mathrm{C}(X^{d^{(B)}_{i_1}}))$ is a map
$$f \mapsto \mathrm{diag}\{\underbrace{f\circ \pi_1, ..., f\circ\pi_{l_{i_1', i_1}}}_{l_{i'_1, i_1}}, \textrm{point evaluations}\},$$ where the point evaluations are arbitrarily chosen (the map $\tilde{\phi}_{i'_1, i_1}^{(A, B)}$ exists by \eqref{divide-cond-n} and \eqref{enough-room-n}) (the evaluation points chosen for the map $\tilde{\phi}_{i'_1, i_1}^{(A, B)}$ might not be suitably dense in $X^{d^{(A)}_{i'_1}}$, but the set of evaluation points of the map  $\tilde{\phi}_{i'_1, i_1}^{(A, B)} \circ \phi^{(A)}_{1, i_1'} $, which contains the set of evaluation points of the map $\phi^{(A)}_{1, i_1'}$, is suitably dense).

Write $$\phi_{1, i_1}^{(A, B)}=\tilde{\phi}_{i'_1, i_1}^{(A, B)} \circ \phi_{1, i'_1}^{(A)},$$ and compress the diagram above as
$$
\xymatrix{
\mathrm{M}_{n_0^{(A)}}(\mathrm{C}(X)) \ar[r]^-{\phi_{1, i_1}^{(A)}} \ar[dr]^-{\phi_{1, i_1}^{(A, B)}} & \mathrm{M}_{m^{(A)}_{i_1}}(\mathrm{C}(X^{d^{(A)}_{i_1}})) \ar[r]^-{\phi_{i_1}^{(A)}}  & \mathrm{M}_{m^{(A)}_{i_1+1}}(\mathrm{C}(X^{d^{(A)}_{i_1+1}})) \ar[r] & \cdots \ar[r] & A \\
\mathrm{M}_{n_0^{(B)}}(\mathrm{C}(X)) \ar[r]^-{\phi_{1, i_1}^{(B)}}  & \mathrm{M}_{m^{(B)}_{i_1}}(\mathrm{C}(X^{d^{(B)}_{i_1}})) \ar[r]^-{\phi_{i_1}^{(B)}}  & \mathrm{M}_{m^{(B)}_{i_1+1}}(\mathrm{C}(X^{d^{(B)}_{i_1+1}})) \ar[r] & \cdots \ar[r] & B.
}
$$
Note that the map $\phi_{i, i_1}^{(A, B)}$ is given by
$$f \mapsto \mathrm{diag}\{\underbrace{f\circ \pi_1, ..., f\circ\pi_{(n_1^{(A)} \cdots n_{i_1'-1}^{(A)}) l_{i_1', i_1}}}_{(n_1^{(A)} \cdots n_{i_1'-1}^{(A)}) l_{i_1', i_1}}, \textrm{point evaluations}\}.$$

Without loss of generality, let us assume that
\begin{equation*}
\delta_2 < \frac{k^{(B)}_{i_1}}{n^{(B)}_{i_1} + k^{(B)}_{i_1}} \quad\mathrm{and}\quad \frac{\frac{3}{4}\delta_2}{1-\frac{3}{4}\delta_2} < \delta_2 <1.
\end{equation*} 
The same argument as above shows that there are $i_2 > i'_2$ such that 
\begin{equation}\label{small-E-2-n}
1 - \prod_{j=0}^\infty \frac{n^{(B)}_{i'_2+j}}{n^{(B)}_{i'_2+j}+k^{(B)}_{i'_2+j}} < \delta_{2},
\end{equation}
\begin{equation}\label{small-c-c-ratio-2} 
\frac{c_{i'_2}^{(B)} \cdots c^{(B)}_{i'_2+j}}{n_{i'_2}^{(B)} \cdots n^{(B)}_{i'_2+j}}( (\frac{n_{i'_2}^{(B)} \cdots n^{(B)}_{i'_2+j}}{c_{i'_2}^{(B)} \cdots c^{(B)}_{i'_2+j}} - 1) +\frac{\delta_2^2}{6}) < \frac{\delta_2^2}{3},\quad j=1, 2, ... ,
\end{equation}
\begin{equation}\label{almost-m-1-A-B}
 \frac{\abs{\{s_k: s_k=1,: k=1, ..., c^{(B)}_{i'_2}\cdots c^{(B)}_{i'_2+j}\}}}{n^{(B)}_{i'_2}\cdots n^{(B)}_{i'_2+j}} > 1-\frac{\delta_2^2}{12},\quad j=1, 2, ...,
\end{equation}
where
$$\phi^{(B)}_{i'_2, i'_2+j} = \mathrm{diag}\{ \underbrace{\underbrace{\pi_1^*, ..., \pi_1^*}_{s_1}, ..., \underbrace{\pi_{c^{(B)}_{i_2'}\cdots c^{(B)}_{i_2'+j}}^*, ..., \pi_{c^{(B)}_{i_2'}\cdots c^{(B)}_{i_2'+j}}^*}_{s_{c^{(B)}_{i_2'}\cdots c^{(B)}_{i_2'+j}}}}_{n^{(B)}_{i_2'}\cdots n^{(B)}_{i'_2+j}},\  \mathrm{point\  evaluations}\},$$

\begin{equation}\label{almost-m-1-B-A}
 \frac{\abs{\{s_k: s_k=1,: k=1, ..., c^{(A)}_{i_2}\cdots c^{(A)}_{i_2+j}\}}}{n^{(A)}_{i_2}\cdots n^{(A)}_{i_2+j}} > 1-\frac{\delta_2^2}{12},\quad j=1, 2, ..., 
\end{equation}
where
$$\phi^{(A)}_{i_2, i_2+j} = \mathrm{diag}\{ \underbrace{\underbrace{\pi_1^*, ..., \pi_1^*}_{s_1}, ..., \underbrace{\pi_{c^{(A)}_{i_2}\cdots c^{(A)}_{i_2+j}}^*, ..., \pi_{c^{(A)}_{i_2}\cdots c^{(A)}_{i_2+j}}^*}_{s_{c^{(A)}_{i_2}\cdots c^{(A)}_{i_2+j}}}}_{n^{(A)}_{i_2}\cdots n^{(A)}_{i_2+j}},\  \mathrm{point\  evaluations}\},$$

\begin{equation}\label{divide-cond-n-2}
n_0^{(A)}\textrm{$\prod_{i=1}^{i_2-1}(n^{(A)}_i+k_{i}^{(A)})$ is divisible by $n_0^{(B)}\prod_{i=1}^{i'_2-1}(n^{(B)}_i+k_{i}^{(B)})$}, 
\end{equation} 
\begin{equation}\label{small-diff-pre-B-A}
\frac{c_{1}^{(B)} \cdots c_{i_2'-1}^{(B)}}{c_{1}^{(A)} \cdots c_{i_2}^{(A)}} < \frac{c_{i_1}^{(B)} \cdots c_{i_2'-1}^{(B)}}{c_{i_1}^{(A)} \cdots c_{i_2-1}^{(A)}} < \frac{\delta_{2}^2}{12},
\end{equation}
and
\begin{equation}\label{enough-room-n-2}
\frac{n_0^{(B)}(n_1^{(B)} + k_1^{(B)}) \cdots (n_{i_2'-1}^{(B)} + k_{i_2'-1}^{(B)}) }{n_0^{(A)}(n_1^{(A)} + k_1^{(A)}) \cdots (n_{i_2'-1}^{(A)} + k_{i_2'-1}^{(A)})} \cdot \frac{l_{i_2', i_2}}{(n_{i'_1}^{(A)} + k_{i'_2}^{(A)}) \cdots (n_{i_2}^{(A)} + k_{i_2}^{(A)})} < 1,
\end{equation}
where
\begin{equation}\label{approx-div-B-A}
0\leq c_1^{(A)} \cdots c_{i_2-1}^{(A)} -  (c_1^{(B)} \cdots c_{i_2'-1}^{(B)}) l_{i'_2, i_2} < c_1^{(B)} \cdots c_{i_2'-1}^{(B)}.
\end{equation}

Consider the map $\tilde{\phi}_{i'_2, i_2}^{(B, A)}: \mathrm{M}_{m^{(B)}_{i'_2}}(\mathrm{C}(X^{d^{(B)}_{i'_2}})) \to \mathrm{M}_{m^{(A)}_{i_2}}(\mathrm{C}(X^{d^{(A)}_{i_1}}))$,
$$f \mapsto \mathrm{diag}\{\underbrace{f\circ \pi_1, ..., f\circ\pi_{l_{i_2', i_2}}}_{l_{i'_2, i_2}}, \textrm{point evaluations}\},$$ where the point evaluations are arbitrarily chosen (the map $\tilde{\phi}_{i'_2, i_2}^{(B, A)}$ exists by \eqref{divide-cond-n-2} and \eqref{enough-room-n-2}; cf.~above).
Define
$$\phi_{i_1, i_2}^{(B, A)}=\tilde{\phi}_{i'_2, i_2}^{(B, A)} \circ \phi_{i_1, i'_2}^{(B)}$$
and consider the augmented diagram 
$$
\xymatrix{
\mathrm{M}_{n_0^{(A)}}(\mathrm{C}(X)) \ar[r]^-{\phi_{1, i_1}^{(A)}} \ar[dr]^-{\phi_{1, i_1}^{(A, B)}} & \mathrm{M}_{m^{(A)}_{i_1}}(\mathrm{C}(X^{d^{(A)}_{i_1}})) \ar[r]^-{\phi_{i_1, i_2}^{(A)}}  & \mathrm{M}_{m^{(A)}_{i_2}}(\mathrm{C}(X^{d^{(A)}_{i_2}})) \ar[r] & \cdots \ar[r] & A \\
\mathrm{M}_{n_0^{(B)}}(\mathrm{C}(X)) \ar[r]^-{\phi_{1, i_1}^{(B)}}  & \mathrm{M}_{m^{(B)}_{i_1}}(\mathrm{C}(X^{d^{(B)}_{i_1}})) \ar[r]^-{\phi_{i_1, i_2}^{(B)}} \ar[ur]^-{\phi_{i_1, i_2}^{(B, A)}}  & \mathrm{M}_{m^{(B)}_{i_2}}(\mathrm{C}(X^{d^{(B)}_{i_2}})) \ar[r] & \cdots \ar[r] & B.
}
$$

It follows from \eqref{small-E-1-n} that
$$\abs{\tau(\phi^{(B, A)}_{i_{1}, i_{2}} \circ \phi^{(A, B)}_{1, i_{1}}(h) - \phi^{(A)}_{i_{1}, i_{2}} \circ \phi^{(A)}_{1, i_1}(h))} < \delta_1,\quad h\in \mathrm{M}_{n_0^{(A)}}(\mathrm{C}(X)),\ \norm{h}\leq 1.$$

Note that
$$\phi^{(B, A)}_{i_{1}, i_{2}} \circ \phi^{(A, B)}_{1, i_{1}} =(\tilde{\phi}_{i'_2, i_2}^{(B, A)} \circ \phi_{i_1, i'_2}^{(B)} \circ \tilde{\phi}_{i'_1, i_1}^{(A, B)}) \circ \phi_{1, i'_1}^{(A)}.$$
By \eqref{almost-m-1-B}, we have
$$\phi_{i_1, i'_2}^{(B)} = \mathrm{diag}\{\underbrace{\pi^*_1, ..., \pi^*_{c_{i_1}^{(B)} \cdots c_{i_2'}^{(B)}}}_{ c_{i_1}^{(B)} \cdots c_{i_2'}^{(B)} }, Q'_0, \textrm{point evaluations} \},$$
where $Q'_0$ is a coordinate projection (possibly with multiplicity) with (in view of \eqref{pert-B-C})
$$\mathrm{rank}(Q'_0) \leq \frac{\delta_1^2}{12}(n_{i_1}^{(B)} \cdots n_{i'_2}^{(B)}) < \frac{\delta_1^2}{6}(c_{i_1}^{(B)} \cdots c_{i'_2}^{(B)}).$$
Hence, 
$$\tilde{\phi}_{i'_2, i_2}^{(B, A)} \circ \phi_{i_1, i'_2}^{(B)} \circ \tilde{\phi}_{i'_1, i_1}^{(A, B)} = \mathrm{diag}\{\underbrace{\pi^*_{1}, ..., \pi^*_{l_{i_1', i_1} (c_{i_1}^{(B)} \cdots c_{i_2'}^{(B)}) l_{i_2', i_2}} }_{l_{i_1', i_1} (c_{i_1}^{(B)} \cdots c_{i_2'}^{(B)}) l_{i_2', i_2}},\ Q_0,\ \textrm{point evaluations} \},$$
where $Q_0$ is a coordinate projection (possibly with multiplicity) with
\begin{equation}\label{size-Q-0}
\mathrm{rank}(Q_0) \leq \frac{\delta_1^2}{6}l_{i_1', i_1}(c_{i_1}^{(B)} \cdots c_{i'_2}^{(B)})l_{i_2', i_2} ,
\end{equation}
and hence,
\begin{equation}\label{A-map-1-2}
\phi^{(B, A)}_{i_{1}, i_{2}} \circ \phi^{(A, B)}_{1, i_{1}} = \mathrm{diag}\{\underbrace{\pi^*_1 \circ P_{1, i_1'}^{(A)}, ..., \pi^*_{l_{i_1', i_1} (c_{i_1}^{(B)} \cdots c_{i_2'}^{(B)}) l_{i_2', i_2}}\circ P_{1, i_1'}^{(A)}}_{(n_1^{(A)} \cdots n_{i_1'-1}^{(A)})l_{i_1', i_1} (c_{i_1}^{(B)} \cdots c_{i_2'}^{(B)}) l_{i_2', i_2}},\  Q_0\circ P_{1, i_1'}^{(A)},\ \tilde{\Theta}'_1\},
\end{equation}
where $P_{1, i_1'}^{(A)}$ is the coordinate projection part of the map $\phi_{1, i_1'}^{(A)}$ and $\tilde{\Theta}_1'$ is a point-evaluation map.
Also note that, by \eqref{almost-m-1-A} (and note that the multiplicities of coordinate projections are non-zero), 
$$\phi_{i_1', i_2}^{(A)} = \mathrm{diag}\{\pi^*_1, ..., \pi^*_{c_{i_1'}^{(A)} \cdots c_{i_2}^{(B)}},\ Q_1,\ \textrm{point evaluations}\},$$
where  $Q_1$ is a coordinate projection (possibly with multiplicity) with
$$\mathrm{rank}(Q_1) \leq \frac{\delta_1^2}{12}(n_{i'_1}^{(A)} \cdots n_{i_2}^{(A)}).$$
We then have
\begin{equation}\label{map-1-2}
 \phi_{1, i_2}^{(A)}  =  \mathrm{diag}\{\underbrace{\pi^*_1\circ P_{1, i_1'}^{(A)}, ...,  \pi^*_{c_{i'_1}^{(A)} \cdots c_{i_2}^{(A)}}\circ P_{1, i_1'}^{(A)}}_{(n_1^{(A)} \cdots n_{i_1'-1}^{(A)}) (c_{i_1'}^{(A)}\cdots  c_{i_2}^{(A)})},\ Q_1\circ P_{1, i_1'}^{(A)},\ \tilde{\Theta}_1''\},
 \end{equation}
 where $\tilde{\Theta}_1''$ is a point-evaluation map.

Let us compare $ \phi^{(B, A)}_{i_{1}, i_{2}} \circ \phi^{(A, B)}_{1, i_{1}} $ (\eqref{A-map-1-2}) with $\phi_{1, i_2}^{(A)} $ (\eqref{map-1-2}). 
Note that

\begin{eqnarray}\label{small-cl}
& & (n_1^{(A)}\cdots n_{i_1'-1}^{(A)} ) l_{i_1', i_1}(c_{i'_1}^{(B)} \cdots c_{i'_2}^{(B)})l_{i_2', i_2} \\
& \leq & (n_1^{(A)}\cdots n_{i_1'-1}^{(A)} ) \frac{c_1^{(B)} \cdots c_{i_1-1}^{(B)}}{c_1^{(A)} \cdots c_{i'_1-1}^{(A)}} (c_{i'_1}^{(B)} \cdots c_{i'_2}^{(B)}) \frac{c_1^{(A)} \cdots c_{i_2-1}^{(A)}}{c_1^{(B)} \cdots c_{i'_2-1}^{(B)}} .  \quad\quad\textrm{(by \eqref{approx-div-A-B} and \eqref{approx-div-B-A})} \nonumber \\
& = & n_1^{(A)}\cdots n_{i_1'-1}^{(A)}  c_{i_1'}^{(A)} \cdots c_{i_2-1}^{(A)} \nonumber \\
& \leq & n_1^{(A)}\cdots n_{i_1'-1}^{(A)}  n_{i_1'}^{(A)} \cdots n_{i_2-1}^{(A)} = n_1^{(A)}\cdots n_{i_2-1}^{(A)}. \nonumber
\end{eqnarray}

Also note that
\begin{eqnarray}\label{pre-small-d-m}
& & c_{1}^{(A)}\cdots c_{i_2}^{(A)} - (c_1^{(A)} \cdots c_{i_1'-1}^{(A)})l_{i'_1, i_1} (c_{i_1}^{(B)} \cdots c_{i_2'}^{(B)}) l_{i_2', i_2} \\
& = & c_1^{(A)} \cdots c_{i_2}^{(A)} - ((c_1^{(A)} \cdots c_{i_1'-1}^{(A)}) l_{i'_1, i_1})(c_{i_1}^{(B)} \cdots c_{i_2'}^{(B)}) l_{i_2', i_2} \nonumber \\
& \leq & c_{1}^{(A)} \cdots c_{i_2}^{(A)} - (c_1^{(B)} \cdots c_{i_1-1}^{(B)}) (c_{i_1}^{(B)} \cdots c_{i_2'}^{(B)}) l_{i_2', i_2} + (c_1^{(A)} \cdots c_{i_1'-1}^{(A)}) (c_{i_1}^{(B)} \cdots c_{i_2'}^{(B)}) l_{i_2', i_2} \quad \textrm{ (\eqref{approx-div-A-B})}  \nonumber \\
&= & c_1^{(A)} \cdots c_{i_2}^{(A)} - (c_1^{(B)} \cdots c_{i_2'-1}^{(B)}) l_{i_2', i_2} + (c_1^{(A)} \cdots c_{i_1'-1}^{(A)}) (c_{i_1}^{(B)} \cdots c_{i_2'-1}^{(B)}) l_{i_2', i_2} \nonumber \\
& \leq & c_1^{(B)} \cdots c_{i_2'-1}^{(B)} + (c_1^{(A)} \cdots c_{i_1'-1}^{(A)}) (c_{i_1}^{(B)} \cdots c_{i_2'-1}^{(B)}) l_{i_2', i_2} \quad\quad \textrm{(by \eqref{approx-div-B-A})}  \nonumber \\
& \leq & c_1^{(B)} \cdots c_{i_2'-1}^{(B)} + (c_1^{(A)} \cdots c_{i_1'-1}^{(A)}) (c_{i_1}^{(B)} \cdots c_{i_2'-1}^{(B)}) \frac{c_1^{(A)} \cdots c_{i_2}^{(A)}}{c_1^{(B)} \cdots c_{i_2'-1}^{(B)}} \quad\quad \textrm{(by \eqref{approx-div-B-A})}  \nonumber \\
& = & c_1^{(B)} \cdots c_{i_2'-1}^{(B)} + (c_1^{(A)} \cdots c_{i_1'-1}^{(A)}) \frac{c_1^{(A)} \cdots c_{i_2}^{(A)}}{c_1^{(B)} \cdots c_{i_1-1}^{(B)}} \nonumber \\
& = & c_1^{(B)} \cdots c_{i_2'-1}^{(B)} + \frac{c_1^{(A)} \cdots c_{i_1'-1}^{(A)}}{c_1^{(B)} \cdots c_{i_1-1}^{(B)}} ({c_1^{(A)} \cdots c_{i_2}^{(A)}}) \nonumber \\
& \leq & \frac{\delta_1^2}{6} ({c_1^{(A)} \cdots c_{i_2}^{(A)}}) \quad\quad \textrm{(by \eqref{small-diff-pre-B-A}, \eqref{small-diff-pre-A-B}),} \nonumber
\end{eqnarray}
and hence,
\begin{eqnarray}\label{small-d-m}
&& \frac{1}{n_{1}^{(A)}\cdots n_{i_2}^{(A)}  }( n_{1}^{(A)}\cdots n_{i_2}^{(A)}  - (n_1^{(A)} \cdots n_{i_1'-1}^{(A)}) l_{i'_1, i_1} (c_{i_1}^{(B)} \cdots c_{i_2'}^{(B)}) l_{i_2', i_2}) \\
& = & 1- \frac{n_1^{(A)} \cdots n_{i_1'-1}^{(A)}}{n_{1}^{(A)}\cdots n_{i_2}^{(A)}  }\cdot l_{i'_1, i_1} (c_{i_1}^{(B)} \cdots c_{i_2'}^{(B)}) l_{i_2', i_2} \nonumber \\
& = & 1-  (\frac{n_1^{(A)} \cdots n_{i_1'-1}^{(A)}}{n_{1}^{(A)}\cdots n_{i_2}^{(A)}  } \cdot \frac{c_1^{(A)} \cdots c_{i_2}^{(A)}}{c_{1}^{(A)}\cdots c_{i'_1-1}^{(A)}  }) \cdot \frac{c_1^{(A)} \cdots c_{i_1'-1}^{(A)}}{c_{1}^{(A)}\cdots c_{i_2}^{(A)}  }l_{i'_1, i_1} (c_{i_1}^{(B)} \cdots c_{i_2'}^{(B)}) l_{i_2', i_2} \nonumber  \\
& = & 1-  (\frac{n_1^{(A)} \cdots n_{i_1'-1}^{(A)}}{c_{1}^{(A)}\cdots c_{i_1'-1}^{(A)}  } \cdot \frac{c_1^{(A)} \cdots c_{i_2}^{(A)}}{n_{1}^{(A)}\cdots n_{i_2}^{(A)}  }) \cdot \frac{c_1^{(A)} \cdots c_{i_1'-1}^{(A)}}{c_{1}^{(A)}\cdots c_{i_2}^{(A)}  }l_{i'_1, i_1} (c_{i_1}^{(B)} \cdots c_{i_2'}^{(B)}) l_{i_2', i_2} \nonumber  \\
& = & 1-  (\frac{c_{i_1'}^{(A)} \cdots c_{i_2}^{(A)}}{n_{i_1'}^{(A)}\cdots n_{i_2}^{(A)}  }) \cdot \frac{c_1^{(A)} \cdots c_{i_1'-1}^{(A)}}{c_{1}^{(A)}\cdots c_{i_2}^{(A)}  }l_{i'_1, i_1} (c_{i_1}^{(B)} \cdots c_{i_2'}^{(B)}) l_{i_2', i_2} \nonumber  \\
&= & \frac{c_{i_1'}^{(A)} \cdots c_{i_2}^{(A)}}{n_{i_1'}^{(A)}\cdots n_{i_2}^{(A)}  } ((\frac{n_{i_1'}^{(A)} \cdots n_{i_2}^{(A)}}{c_{i_1'}^{(A)}\cdots c_{i_2}^{(A)}  } - 1) +(1-  \frac{c_1^{(A)} \cdots c_{i_1'-1}^{(A)}}{c_{1}^{(A)}\cdots c_{i_2}^{(A)}  }l_{i'_1, i_1} (c_{i_1}^{(B)} \cdots c_{i_2'}^{(B)}) l_{i_2', i_2})) \nonumber  \\
& \leq & \frac{c_{i_1'}^{(A)} \cdots c_{i_2}^{(A)}}{n_{i_1'}^{(A)}\cdots n_{i_2}^{(A)}  } ((\frac{n_{i_1'}^{(A)} \cdots n_{i_2}^{(A)}}{c_{i_1'}^{(A)}\cdots c_{i_2}^{(A)}  } - 1) +\frac{\delta_1^2}{6}) \quad\quad \textrm{(by \eqref{pre-small-d-m})}   \nonumber \\
&< &\frac{\delta_1^2}{3} \quad\quad \textrm{(by \eqref{small-c-c-ratio}).} \nonumber 
\end{eqnarray}

Then,
\begin{eqnarray*}
&& (n_1^{(A)} \cdots n_{i_1'-1}^{(A)}) (c_{i_1'}^{(A)}\cdots  c_{i_2}^{(A)}) - (n_1^{(A)} \cdots n_{i_1'-1}^{(A)})l_{i_1', i_1} (c_{i_1}^{(B)} \cdots c_{i_2'}^{(B)}) l_{i_2', i_2} \\
& \leq & (n_1^{(A)} \cdots n_{i_1'-1}^{(A)}) (n_{i_1'}^{(A)}\cdots  n_{i_2}^{(A)}) - (n_1^{(A)} \cdots n_{i_1'-1}^{(A)})l_{i_1', i_1} (c_{i_1}^{(B)} \cdots c_{i_2'}^{(B)}) l_{i_2', i_2} \\
&\leq & \frac{\delta_1^2}{3}(n_{1}^{(A)}\cdots n_{i_2}^{(A)} ), 
\end{eqnarray*}
and hence,
\begin{equation}
(c_{i_1'}^{(A)}\cdots  c_{i_2}^{(A)}) -l_{i_1', i_1} (c_{i_1}^{(B)} \cdots c_{i_2'}^{(B)}) l_{i_2', i_2} \leq  \frac{\delta_1^2}{3} \cdot \frac{n_{1}^{(A)}\cdots n_{i_2}^{(A)} }{n_1^{(A)} \cdots n_{i_1'-1}^{(A)}}. 
\end{equation}

Write
$$ 
P_1 = \mathrm{diag}\{\underbrace{\pi^*_1\circ P_{1, i_1'}^{(A)}, ..., \pi^*_{l_{i_1', i_1} (c_{i_1}^{(B)} \cdots c_{i_2'}^{(B)}) l_{i_2', i_2}}\circ P_{1, i_1'}^{(A)}}_{(n_1^{(A)} \cdots n_{i_1'-1}^{(A)})l_{i_1', i_1} (c_{i_1}^{(B)} \cdots c_{i_2'}^{(B)}) l_{i_2', i_2}}\}.
$$
Note that
\begin{eqnarray*}
&& \abs{\mathrm{rank}(\tilde{\Theta}_1')- \mathrm{rank}(\tilde{\Theta}_1'') } \\
&\leq & (n_1^{(A)}\cdots n_{i_1'-1}^{(A)} )( (c_{i_1'}^{(A)}\cdots  c_{i_2}^{(A)}) -l_{i_1', i_1} (c_{i_1}^{(B)} \cdots c_{i_2'}^{(B)}) l_{i_2', i_2}  + \mathrm{rank}(Q_0) + \mathrm{rank}(Q_1)) \\
&\leq &(n_1^{(A)}\cdots n_{i_1'-1}^{(A)} )(\frac{\delta_1^2}{3} \cdot \frac{n_{1}^{(A)}\cdots n_{i_2}^{(A)} }{n_1^{(A)} \cdots n_{i_1'-1}^{(A)}} + \frac{\delta_1^2}{6} l_{i_1', i_1}(c_{i'_1}^{(B)} \cdots c_{i'_2}^{(B)})l_{i_2', i_2} + \frac{\delta_1^2}{12}(n_{i'_1}^{(A)} \cdots n_{i_2}^{(A)})) \\
&= &\frac{5}{12}\delta_1^2 ( n_{1}^{(A)}\cdots n_{i_2}^{(A)}  ) + \frac{\delta_1^2}{6}(n_1^{(A)}\cdots n_{i_1'-1}^{(A)} ) l_{i_1', i_1}(c_{i'_1}^{(B)} \cdots c_{i'_2}^{(B)})l_{i_2', i_2} \\
& < & \frac{5}{12}\delta_1^2 ( n_{1}^{(A)}\cdots n_{i_2}^{(A)}  ) + \frac{\delta_1^2}{6}(n_1^{(A)}\cdots n_{i_2}^{(A)})\quad\quad \textrm{by \eqref{small-cl}}\\
&= &\frac{7}{12}\delta_1^2  ( n_{1}^{(A)}\cdots n_{i_2}^{(A)}  ). 
\end{eqnarray*}
Then, $\tilde{\Theta}_1'$ and $\tilde{\Theta}_1''$ can be decomposed as 
$$ \tilde{\Theta}_1' = \tilde{R}_1'  \oplus \Theta'_1 \quad \mathrm{and} \quad  \tilde{\Theta}_1'' = \tilde{R}_1'' \oplus \Theta''_1,$$
with 
\begin{equation}\label{size-theta}
\mathrm{rank}(\Theta'_1) = \mathrm{rank}(\Theta''_1) \quad \mathrm{and}\quad \max\{\mathrm{rank}(\tilde{R}_1'), \mathrm{rank}(\tilde{R}_1'')\} \leq \frac{7}{12}{\delta_1^2}  ( n_{1}^{(A)}\cdots n_{i_2}^{(A)}). 
\end{equation}

Define
$$
R_1'=\mathrm{diag}\{P_{1, i_1'}^{(A)}\circ Q_0,\ \tilde{R}'_1\}
$$
and
$$
R_1'' = \mathrm{diag}\{\pi^*_{{l_{i_1', i_1} (c_{i_1}^{(B)} \cdots c_{i_2'}^{(B)}) l_{i_2', i_2}}+1}\circ P_{1, i_1'}^{(A)}, ..., \pi^*_{c_{i'_1}^{(A)} \cdots c_{i_2}^{(A)}}\circ P_{1, i_1'}^{(A)},\ Q_1\circ P_{1, i_1'}^{(A)},\ \tilde{R}_1''\}.
$$
Then we have
$$\phi^{(B, A)}_{i_1, i_2} \circ \phi^{(A, B)}_{1, i_1} =\mathrm{diag}\{P_1, R'_1, \Theta'_1 \} $$ 
and
$$\phi^{(A)}_{i_1, i_2} \circ \phi^{(A)}_{1, i_1} = \mathrm{diag}\{P_1, R''_1,  \Theta_1''\}, $$
with (by \eqref{size-Q-0}, \eqref{size-theta}, and \eqref{small-cl})
\begin{eqnarray}\label{small-rank-leftover}
\mathrm{rank}(R_1'') = \mathrm{rank}(R_1') & < & \frac{\delta_1^2}{6} (n_1^{(A)}\cdots n_{i_1'-1}^{(A)} ) l_{i_1', i_1}(c_{i_1}^{(B)} \cdots c_{i'_2}^{(B)})l_{i_2', i_2} + \frac{7}{12}{\delta_1^2}  ( n_{1}^{(A)}\cdots n_{i_2}^{(A)}) \nonumber \\
& < & \frac{3\delta_1^2}{4}( n_{1}^{(A)}\cdots n_{i_2}^{(A)} ).
\end{eqnarray}

Note that $\mathrm{rank}(\tilde{\Theta}_1'')$---the number of the point evaluations appearing in $\phi_{i_1, i_2}^{(A)} \circ \phi_{1, i_1}^{(A)}$---is at least
$$ \frac{k^{(A)}_1}{n^{(A)}_1 + k^{(A)}_1}((n_1^{(A)}+ k_1^{(A)})\cdots (n_{i_2}^{(A)} + k_{i_2}^{(A)})),$$ and hence, by \eqref{delta-assumption}, is at least $$\delta_1 (n_1^{(A)}\cdots n_{i_2}^{(A)}).$$
It then follows from \eqref{small-rank-leftover} that
$$\mathrm{rank}(\Theta_1') = \mathrm{rank}(\Theta_1'') >  (\delta_1-\frac{3\delta_1^2}{4}) (n_1^{(A)}\cdots n_{i_2}^{(A)}), $$
and hence (by \eqref{small-rank-leftover} again), that 
$$\frac{ \mathrm{rank}(R'_1)} {\mathrm{rank}(\Theta'_1)}  = \frac{ \mathrm{rank}(R''_1)} {\mathrm{rank}(\Theta''_1)} < \frac{\frac{3}{4}\delta_1^2}{\delta_1-\frac{3}{4} \delta_1^2 } <  \delta_1.$$

Repeating this process, we have an intertwining diagram which is approximately commutative in the sense desired.
\end{proof}

\subsection{The isomorphism theorem}
First, we need the following stable uniqueness theorem, which certainly is well known to experts (see, for instance, \cite{DL-classification}, \cite{Lin-s-uniq}, and \cite{GL-almost-map}). For the reader's convenience, we provide a proof. 
\begin{thm}\label{stable-uniq}
Let $X$ be a K-contractible metrizable compact space (i.e., $\Kzero(\mathrm{C}(X)) = \Int$ and $\Kone(\mathrm{C}(X)) = \{0\}$), and let $\Delta: \mathrm{C}(X)^+ \to (0, +\infty)$ be a map which preserves order. For any finite set $\mathcal F\subseteq \mathrm{C}(X)$ and any $\eps>0$, there exists a finite set $\mathcal H\subseteq \mathrm{C}(X)^+$ with $\mathrm{supp}(h) \neq X$ for each $h\in\mathcal H$ and there exists $M\in \mathbb N$ such that the following property holds:
for any unital homomorphisms $$\phi, \psi: \mathrm{C}(X) \to \mathrm{M}_{n}(\mathrm{C}(Y))\quad \textrm{and}\quad \theta: \mathrm{C}(X) \to \mathrm{M}_{m}(\Comp) \subseteq \mathrm{M}_{m}(\mathrm{C}(Y)),$$
where $\theta$ is a unital point-evaluation map with $nM < m,$ and such that
$$\mathrm{tr}(\theta(h)) > \Delta(h),\quad h\in\mathcal H,$$
there is a unitary $u\in\mathrm{M}_{n+m}(\mathrm{C}(Y))$ such that
$$\norm{\mathrm{diag}\{\phi(a), \theta(a)\} - u^*\mathrm{diag}\{\psi(a), \theta(a)\} u} < \eps,\quad a\in\mathcal F.$$ 
\end{thm}

The theorem follows from the following two lemmas.

\begin{lem}\label{pre-stable-uniq}
Let $X$ be a K-contractible metrizable compact space, and let $\Delta: \mathrm{C}(X)^+ \to (0, +\infty)$ be an order-preserving map. For any finite set $\mathcal F\subseteq \mathrm{C}(X)$ and any $\eps>0$, there exists a finite set $\mathcal H\subseteq \mathrm{C}(X)^+$ with $\mathrm{supp}(h) \neq X$ for each $h\in\mathcal H$ and there exists $M\in \mathbb N$ such that the following property holds:
for any unital homomorphisms $$\phi, \psi: \mathrm{C}(X) \to \mathrm{M}_{n}(\mathrm{C}(Y))\quad \textrm{and}\quad \theta: \mathrm{C}(X) \to \mathrm{M}_{n}(\Comp) \subseteq \mathrm{M}_{n}(\mathrm{C}(Y))$$
where $\theta$ is a unital point-evaluation map with 
$$\mathrm{tr}(\theta(h)) > \Delta(h),\quad h\in\mathcal H,$$
there is a unitary $u\in\mathrm{M}_{(1+M)n}(\mathrm{C}(Y))$ such that
$$\norm{\mathrm{diag}\{\phi(a), \underbrace{\theta(a), ..., \theta(a)}_M\} - u^*\mathrm{diag}\{\psi(a), \underbrace{\theta(a), ..., \theta(a)}_M \}u} < \eps,\quad a\in\mathcal F.$$ 
\end{lem}

\begin{proof}
Assume the statement were not true. Then there would be $(\mathcal F_0, \eps_0)$ such that for any finite set $\mathcal H\subseteq \mathrm{C}(X)^+$ and any $M$, there are unital homomorphisms $\phi, \psi, \theta: \mathrm{C}(X) \to  \mathrm{M}_n(\mathrm{C}(Y))$ for some $Y$ and $n$ with $\theta$ a unital point-evaluation map with $$\mathrm{tr}(\theta(h)) > \Delta(h),\quad h\in\mathcal H,$$ but $$\norm{\mathrm{diag}\{\phi(a), \underbrace{\theta(a), ..., \theta(a)}_M\} - u^*\mathrm{diag}\{\psi(a), \underbrace{\theta(a), ..., \theta(a)}_M \}u} \geq \eps,\quad a\in\mathcal F .$$

In particular, let $\mathcal H_i$, $i=1, 2, ...$, be an increasing sequence of finite sets with dense union in $\mathrm{C}(X)^+$, and $\mathrm{supp}(h) \neq X$ for each $h\in\mathcal H_i$, $i=1, 2, ...$. There are sequences of unital homomorphisms $\phi_i, \psi_i, \theta_i: \mathrm{C}(X) \to  \mathrm{M}_{n_i}(\mathrm{C}(Y_i))=:B_i$ for some $Y_i$ and $n_i$ with $\theta_i$ a unital point-evaluation map with 
\begin{equation}\label{fullness}
\mathrm{tr}(\theta_i(h)) > \Delta(h),\quad h\in\mathcal H_i,
\end{equation} 
but 
\begin{equation}\label{cont-assumption}
\norm{\mathrm{diag}\{\phi_i(a), \underbrace{\theta_i(a), ..., \theta_i(a)}_i\} - u^*\mathrm{diag}\{\psi_i(a), \underbrace{\theta_i(a), ..., \theta_i(a)}_i \}u} \geq \eps_0,\quad a\in\mathcal F_0,
\end{equation}
for all unitary $u \in \mathrm{M}_{1+i}(B_i)$.

Consider the three maps $\Phi:=(\phi_i), \Psi:=(\psi_i), \Theta:=(\theta_i): \mathrm{C}(X) \to \prod B_i/\bigoplus B_i$. 
Since $X$ is K-contractible, by the UCT (\cite{RS-UCT}), the C*-algebra $\mathrm{C}(X)$ is KK-equivalent to $\Comp$, and hence we have 
\begin{equation}\label{KK-eq}
[\Phi] = [\Psi]\quad\mathrm{in}\ \mathrm{KK}(\mathrm{C}(X), \prod B_i/\bigoplus B_i).
\end{equation}

By \eqref{fullness}, the map $\Theta$ is a unital full embedding (see Definition 2.8 of \cite{DL-classification}; we leave the verification to the reader), and then, by Theorem 2.22 together with Theorem 4.5 of \cite{DL-classification}, there exist $l\in\mathbb N$ and a unitary $u\in \mathrm{M}_{1+l}(\prod B_i/\bigoplus B_i)$ such that
$$\norm{\mathrm{diag}\{\Phi(a), \underbrace{\Theta(a), ..., \Theta(a)}_l\} - u^*\mathrm{diag}\{\Psi(a), \underbrace{\Theta(a), ..., \Theta(a)}_l \}u} < \eps_0,\quad a\in\mathcal F_0 .$$ Lifting $u$ to a unitary of $\prod B_i$ (the relations for a unitary are stable), one has a contradiction with \eqref{cont-assumption}.
\end{proof}

\begin{lem}\label{div-M}
Let $X$ be a compact metric space, and let $\Delta: \mathrm{C}(X)^+ \to (0, +\infty)$ be a density function (i.e., an order-preserving map). For any finite set $\mathcal F\subseteq \mathrm{C}(X)$, any $\eps>0$, and any $M\in \mathbb N$, there exist a finite set $\mathcal H\subseteq \mathrm{C}(X)^+$ and $L \in \mathbb N$, with $\mathrm{supp}(h) \neq X$ for each $h\in\mathcal H$, such that if $\theta: \mathrm{C}(X) \to \mathrm{M}_n(\Comp)$, where $n>L$, is a point-evaluation map with
$$\mathrm{tr}(\theta(h)) > \Delta(h),\quad h\in\mathcal H,$$ then there are unital homomorphisms $\theta_0: \mathrm{C}(X) \to \mathrm{M}_{n_0}(\Comp)$ and  $\theta_1: \mathrm{C}(X) \to \mathrm{M}_{n_1}(\Comp)$ for some $n_0, n_1$ with $n_0 +Mn_1 = n$ and a (permutation) unitary $u$ such that
$$\norm{ \theta(a) - u^* (\theta_0(h) \oplus \underbrace{\theta_1(a)\oplus \cdots \theta_1(a)}_M) u} < \eps,\quad a\in\mathcal F,$$
and
$$n_0 \leq n_1.$$
\end{lem}

\begin{proof}
Pick $\delta>0$ such that $\mathrm{dist}(x, y) < \delta$ implies $\abs{f(x) - f(y)} < \eps$ for all $f\in\mathcal F$. Pick a finite set $\{y_1, ..., y_s\}\subseteq X$ which is $\delta$-dense in $X$. Define $$\delta_0:=\min\{\mathrm{dist}(y_i, y_j): i\neq j, i, j =1, ..., s\},$$ and pick non-zero continuous functions $h_i: X \to [0, 1]$, $i=1, ..., s$, such that 
$$
h_i(x) = 0, \quad \mathrm{if}\ \mathrm{dist}(x, y_i) > \delta_0/2.
$$

Set $\mathcal H = \{h_1, ..., h_s\}$, and pick an integer $$L > \frac{M^2+M}{\min\{\Delta(h_i): i=1, ..., s\}}.$$ Then $\mathcal H$ and $L$ have the property of the lemma.

Indeed, let $\theta: \mathrm{C}(X)\to \mathrm{M}_n$ be a point-evaluation map satisfying 
\begin{equation}\label{dense-ev}
\mathrm{tr}(\theta(h)) > \Delta(h),\quad h\in\mathcal H.
\end{equation} 
Write $\{x_1, ..., x_n\}$ for the evaluation points of $\theta$. Then, choose a map $\sigma: \{x_1, ..., x_n\} \to \{y_1, ..., y_s\}$ such that
\begin{equation}\label{small-move}
\mathrm{dist}(x_i, \sigma(x_i)) < \delta,\quad i=1, ..., L,
\end{equation}
and for each $j=1, ..., s$,
$$ \sigma(x_i) = y_j\quad \textrm{if $\mathrm{dist}(x_i, y_j)<\delta_0/2$}.$$ Let $\theta': \mathrm{C}(X) \to \mathrm{M}_n(\Comp)$ denote  the point-evaluation map at the points $\sigma(x_1), ..., \sigma(x_n)$. Then, it follows from \eqref{small-move} and the choice of $\delta$ that
\begin{equation}\label{new-ev}
\norm{\theta(f) - \theta'(f)} < \eps,\quad f\in\mathcal F.
\end{equation} 
Up to a permutation, write $$\theta' = \mathrm{diag}\{ \underbrace{\mathrm{ev}_{y_1}, ..., \mathrm{ev}_{y_1}}_{m_1}, ..., \underbrace{\mathrm{ev}_{y_s}, ..., \mathrm{ev}_{y_s}}_{m_s}\}.$$
Note that
$$\mathrm{tr}(\theta)(h_i) \leq \frac{m_i}{n},\quad i=1, ..., s.$$
By \eqref{dense-ev}, we have 
$$m_i \geq n \Delta(h_i) \geq L\Delta(h_i) > M^2+M.$$ Then, write $m_i = Md_i +r_i$ with $0\leq r_i \leq M-1$, so that, in particular, $d_i> r_i$.  Write
$$\theta_0= \mathrm{diag}\{ \underbrace{\overbrace{\mathrm{ev}_{y_1}, ..., \mathrm{ev}_{y_1}}^{r_1}, 0, ..., 0}_{m_1}, ..., \underbrace{\overbrace{\mathrm{ev}_{y_s}, ..., \mathrm{ev}_{y_s}}^{r_s}, 0, ..., 0}_{m_s}\}$$
and
$$\theta_1= \mathrm{diag}\{\underbrace{\overbrace{0, ..., 0}^{r_1}, \overbrace{\mathrm{ev}_{y_1}, ..., \mathrm{ev}_{y_1}}^{d_1}, 0, ..., 0}_{m_1}, ..., \underbrace{\overbrace{0, ..., 0}^{r_s}, \overbrace{\mathrm{ev}_{y_s}, ..., \mathrm{ev}_{y_s}}^{d_s}, 0, ..., 0}_{m_s}\}.$$
A straightforward calculation shows that
$$ \theta'= \theta_0(h) \oplus \underbrace{\theta_1(a)\oplus \cdots \theta_1(a)}_M,$$
and the desired inequality is then just \eqref{new-ev}.
\end{proof}

\begin{proof}[Proof of Theorem \ref{stable-uniq}]
Applying Lemma \ref{pre-stable-uniq} to $(\mathcal F, \eps/2)$ with respect to the density function $\Delta/4$, we obtain $\mathcal H_1\subseteq \mathrm C(X)^+$ and $M_1$. Applying Lemma \ref{div-M} to $(\mathcal F \cup\mathcal H_1,\ \min\{\eps/4, \Delta(h)/4: h\in\mathcal H_1\})$ and $2M_1$ with respect to the density function $\Delta$, we obtain $\mathcal H_2$ and $L$. Then $\mathcal H:=\mathcal H_1\cup\mathcal H_2$ and $M:=2LM_1$ satisfy the condition of the theorem.

Indeed, let $\phi, \psi: \mathrm{C}(X) \to \mathrm{M}_n(\mathrm{C}(Y))$ and $\theta: \mathrm{C}(X) \to \mathrm{M}_m(\mathrm{C}(Y))$ be unital homomorphisms such that $\theta$ is a unital point-evaluation map with $n(2LM_1) < m$ (in particular, $L<m$ and $n< m/2M_1$), and 
$$\mathrm{tr}(\theta(h)) > \Delta(h),\quad h\in\mathcal H=\mathcal H_1\cup\mathcal H_2.$$
By Lemma \ref{div-M}, there are unital homomorphisms $\theta_0: \mathrm{C}(X)\to \mathrm{M}_{n_0}(\Comp)$ and
$\theta_1: \mathrm{C}(X)\to \mathrm{M}_{n_1}(\Comp)$ for some $n_0$, $n_1$ with $n_0+2M_1n_1 = m$ and a permutation unitary $u \in \mathrm{M}_m(\Comp)$ 
such that
$$\norm{ \theta(a) - u^* (\theta_0(h) \oplus \underbrace{\theta_1(a)\oplus \cdots \theta_1(a)}_{2M_1}) u} <  \min\{\eps/4, \Delta(h)/4: h\in\mathcal H_1\},\quad a\in\mathcal F \cup\mathcal H_1,$$
and
$$n_0 \leq n_1.$$
In particular,
$$ \norm{(1_n\oplus u)(\phi(a)\oplus\theta(a))(1_n\oplus u^*) - \phi(a)\oplus\theta_0(a)\oplus\underbrace{\theta_1(a)\oplus\cdots\oplus\theta_1(a)}_{2M_1} } < \frac{\eps}{4},\quad a\in \mathcal F$$
and
$$ \norm{(1_n\oplus u)(\psi(a)\oplus\theta(a))(1_n\oplus u^*) - \phi(a)\oplus\theta_0(a)\oplus\underbrace{\theta_1(a)\oplus\cdots\oplus\theta_1(a)}_{2M_1} } < \frac{\eps}{4},\quad a\in \mathcal F.$$

Also note that 
$$\mathrm{tr}(\theta_1(h)) > \frac{1}{4}\Delta(h),\quad h \in \mathcal H_1.$$
Now, consider the maps
$$(\phi \oplus \theta_0) \oplus (\underbrace{\theta_1 \oplus \cdots \oplus \theta_1}_{2M_1}) \quad \mathrm{and}\quad (\psi \oplus \theta_0) \oplus (\underbrace{\theta_1 \oplus \cdots \oplus \theta_1}_{2M_1}).$$
Then it follows from (the conclusion of) Lemma \ref{pre-stable-uniq} that there is a unitary $v \in \mathrm{M}_{n+m}(\mathrm{C}(Y))$ such that
$$ \norm{ (\phi(a) \oplus \theta_0(a)) \oplus (\underbrace{\theta_1(a) \oplus \cdots \oplus \theta_1(a)}_{2M_1}) - v^*((\psi(a) \oplus \theta_0(a)) \oplus (\underbrace{\theta_1(a) \oplus \cdots \oplus \theta_1(a)}_{2M_1}))v } < \frac{\eps}{2},\quad a\in \mathcal F.$$
Therefore
$$\norm{\phi(a) \oplus \theta(a) - (1_n\oplus u^*)v^*(1_n \oplus u)(\psi(a)\oplus \theta(a)) (1_n\oplus u)v(1_n \oplus u)} < \eps,\quad a\in\mathcal F,$$
as desired.
\end{proof}

\begin{proof}[Proof of Theorem \ref{n-theorm}]
If $\mathrm{rc}(A) = \mathrm{rc}(B) = 0$, then $A$ and $B$ are $\mathcal Z$-stable, and hence $A\cong B$ if, and only if, $(\Kzero(A), \mathrm{T}(A)) \cong (\Kzero(B), \mathrm{T}(B))$ (note that $\Kzero(A)$ and $\Kzero(B)$ have a unique state, and hence the pairing maps are automatically isomorphic if $\mathrm{T}(A) \cong \mathrm{T}(B)$). 

Now, let us assume $\mathrm{rc}(A) = \mathrm{rc}(B) \neq 0$. Since $X$ is solid,  by Theorem \ref{rcA}, 
$$
\frac{\mathrm{dim}(X)}{n_0^{(A)}}\prod_{i=1}^\infty \frac{c^{(A)}_i}{n^{(A)}_i + k^{(A)}_i} = \frac{\mathrm{dim}(X)}{n_0^{(B)}} \prod_{i=1}^\infty \frac{c^{(B)}_i}{n^{(B)}_i + k^{(B)}_i}.
$$
Since $\mathrm{dim}(X) < \infty$, both sides are finite non-zero numbers, and hence
$$
\frac{1}{n_0^{(A)}}\prod_{i=1}^\infty \frac{c^{(A)}_i}{n^{(A)}_i + k^{(A)}_i} = \frac{1}{n_0^{(B)}} \prod_{i=1}^\infty \frac{c^{(B)}_i}{n^{(B)}_i + k^{(B)}_i} \neq 0;
$$
that is,  \eqref{same-rc} of Lemma \ref{trace-twining-2} is satisfied. Note that \eqref{super-n} of Lemma \ref{trace-twining-2} follows from the assumption $\Kzero(A) \cong \Kzero(B)$.

Consider the inductive limit constructions
$$
\xymatrix{
\mathrm{M}_{n_0^{(A)}}(\mathrm{C}(X)) \ar[r]^-{\phi_{1}^{(A)}} & \mathrm{M}_{m_2^{(A)}}(\mathrm{C}(X^{d^{(A)}_2})) \ar[r]^-{\phi_{2}^{(A)}}  & \mathrm{M}_{m_3^{(A)}}(\mathrm{C}(X^{d^{(A)}_3})) \ar[r] & \cdots \ar[r] & A, \\
\mathrm{M}_{n_0^{(B)}}(\mathrm{C}(X)) \ar[r]^-{\phi_{1}^{(B)}}  & \mathrm{M}_{m_2^{(B)}}(\mathrm{C}(X^{d^{(B)}_2})) \ar[r]^-{\phi_{2}^{(F)}}  & \mathrm{M}_{m_3^{(B)}}(\mathrm{C}(X^{d^{(B)}_3})) \ar[r] & \cdots \ar[r] & B,
}
$$
where
$$m_i:=n_0(n_1+k_1)\cdots (n_{i-1}+k_{i-1}), \quad d_i:=c_1\cdots c_{i-1}.$$

Choose finite subsets $$\mathcal F^{(A)}_1\subseteq \mathrm{M}_{n_0^{(A)}}(C(X)), \mathcal F^{(A)}_2\subseteq \mathrm{M}_{n_0^{(A)}(n^{(A)}_1+k^{(A)}_1)}(\mathrm{C}(X^{n^{(A)}_1})), ... $$ 
and 
$$\mathcal F^{(B)}_1\subseteq \mathrm{M}_{n_0^{(B)}}(C(X)), \mathcal F^{(B)}_2\subseteq \mathrm{M}_{n_0^{(B)}(n^{(B)}_1+k^{(B)}_1)}(\mathrm{C}(X^{n_1^{(B)}})), ... $$ such that
$$\overline{\bigcup_{i=1}^\infty \mathcal F_i^{(A)}} = A \quad \mathrm{and}\quad \overline{\bigcup_{i=1}^\infty \mathcal F_i^{(B)}} = B.$$
Also choose
$\eps_1 > \eps_2 > \cdots > 0$
such that
$$\sum_{i=1}^\infty \eps_i \leq 1.$$

Since $A$ and $B$ are simple, we have the non-zero (order-preserving) density functions
$$\Delta_A(h) = \inf\{\tau(h): \tau\in\mathrm{T}(A)\}, \quad h\in A^+, $$ 
and  
$$ \Delta_B(h) = \inf\{\tau(h): \tau\in\mathrm{T}(B)\},\quad h\in B^+.$$

For $i=0, 2, ...$, applying Theorem \ref{unique} to $(\mathcal F_i^{(A)}, \eps_i/2)$ with respect to $\Delta_A/4$, we obtain finite sets $\mathcal H_{i, 0}^{(A)}, \mathcal H_{i, 1}^{(A)} \subseteq \mathrm{M}_{m^{(A)}_i}(\mathrm{C}(X^{d^{(A)}_i}))$ and $\delta^{(A)}_i>0$. Applying Theorem \ref{stable-uniq} to $(\mathcal F_i^{(A)}, \eps_i/2)$  with respect to $\Delta_A/2$, we obtain $M^{(A)}_i > 0$ and a finite set $\mathcal H_{i, 2}^{(A)} \subseteq \mathrm{M}_{m^{(A)}_i}(\mathrm{C}(X^{d^{(A)}_i}))$.

For $i=1, 3, ...$, applying Theorem \ref{unique} to $(\mathcal F_i^{(B)}, \eps_i/2)$ with respect to $\Delta_B/4$, we obtain finite sets $\mathcal H_{i, 0}^{(B)}, \mathcal H_{i, 1}^{(B)} \subseteq \mathrm{M}_{m^{(B)}_i}(\mathrm{C}(X^{d^{(B)}_i}))$ and $\delta^{(B)}_i>0$. Applying Theorem \ref{stable-uniq} to $(\mathcal F_i^{(B)}, \eps_i/2)$  with respect to $\Delta_B/2$, we obtain $M^{(B)}_i > 0$ and a finite set $\mathcal H_{i, 2}^{(B)} \subseteq \mathrm{M}_{m^{(B)}_i}(\mathrm{C}(X^{d^{(B)}_i}))$. 

For each $i=1, 2, ...$, set $$\mathcal H_i^{(A)} = \mathcal H_{i, 0}^{(A)} \cup \mathcal H_{i, 1}^{(A)} \cup \mathcal H_{i, 2}^{(A)},\quad \mathcal H_i^{(B)} = \mathcal H_{i, 0}^{(B)} \cup \mathcal H_{i, 1}^{(B)} \cup \mathcal H_{i, 2}^{(B)},$$
and
$$\delta_i = \min\{\frac{\delta_i^{(A)}}{2}, \frac{\delta_i^{(B)}}{2}, \frac{1}{M_i^{(A)}}, \frac{1}{M_i^{(B)}}, \frac{1}{4}\Delta_A(h_A), \frac{1}{4}\Delta_B(h_B): h_A\in \mathcal H_i^{(A)}, h_B\in\mathcal H_i^{(B)}\}.$$

After telescoping, we may assume that
\begin{equation}\label{bound-tr-A}
\tau(\phi^{(A)}_{i, i+1}(h)) > \Delta_A(h)/2,\quad h\in\mathcal H^{(A)}_{i},\  \tau\in\mathrm{T}(\mathrm{M}_{m_{i+1}^{(A)}}(\mathrm{C}(X^{d_{i+1}^{(A)}}))),\ i=0, 1, ... 
\end{equation}
and
\begin{equation}\label{bound-tr-B}
 \tau(\phi^{(B)}_{i, i+1}(h)) > \Delta_B(h)/2,\quad h\in\mathcal H^{(B)}_{i},\  \tau\in\mathrm{T}(\mathrm{M}_{m_{i+1}^{(B)}}(\mathrm{C}(X^{d_{i+1}^{(B)}}))),\ i=0, 1, .... 
 \end{equation}

By Lemma \ref{trace-twining-2}, there is a diagram
\begin{equation}\label{diag-0-7-1}
\xymatrix{
\mathrm{M}_{n_0^{(A)}}(\mathrm{C}(X)) \ar[r]^-{\phi_{1, i_1}^{(A)}} \ar[dr]^-{\phi_{1, i_1}^{(A, B)}} & \mathrm{M}_{m^{(A)}_{i_1}}(\mathrm{C}(X^{d^{(A)}_{i_1}})) \ar[r]^-{\phi_{i_1, i_2}^{(A)}}  & \mathrm{M}_{m^{(A)}_{i_2}}(\mathrm{C}(X^{d^{(A)}_{i_2}})) \ar[r] \ar[dr]^-{\phi_{i_2, i_3}^{(A, B)}} & \cdots \ar[r] & A \\
\mathrm{M}_{n_0^{(B)}}(\mathrm{C}(X)) \ar[r]^-{\phi_{1, i_1}^{(B)}}  & \mathrm{M}_{m^{(B)}_{i_1}}(\mathrm{C}(X^{d^{(B)}_{i_1}})) \ar[r]^-{\phi_{i_1, i_2}^{(B)}} \ar[ur]^-{\phi_{i_1, i_2}^{(B, A)}}  & \mathrm{M}_{m^{(F)}_{i_2}}(\mathrm{C}(X^{d^{(B)}_{i_2}})) \ar[r] & \cdots \ar[r] & B
}
\end{equation}
such that
$$\abs{\tau(\phi^{(B, A)}_{i_{s+1}, i_{s+2}} \circ \phi^{(A, B)}_{i_s, i_{s+1}}(h) - \phi^{(A)}_{i_{s+1}, i_{s+2}} \circ \phi^{(A)}_{i_s, i_{s+1}}(h))} < \delta_{i_s}$$
for any $s=0, 2, ...$, any $h\in \mathrm{M}_{m^{(A)}_{i_s}}(\mathrm{C}(X^{d^{(A)}_{i_s}}))$ with $\norm{h}\leq 1$, and any $\tau\in\mathrm{T}(\mathrm{M}_{m^{(A)}_{i_{s+2}}}(\mathrm{C}(X^{d^{(A)}_{i_{s+2}}})))$, 
and, furthermore,
$$\abs{\tau(\phi^{(A, B)}_{i_{s+1}, i_{s+2}} \circ \phi^{(B, A)}_{i_s, i_{s+1}}(h) - \phi^{(B)}_{i_{s+1}, i_{s+2}} \circ \phi^{(B)}_{i_s, i_{s+1}}(h))} < \delta_{i_s}$$
for any $s=1, 3, ...$, any $h\in \mathrm{M}_{m^{(B)}_{i_s}}(\mathrm{C}(X^{d^{(B)}_{i_s}}))$ with $\norm{h}\leq 1$, and any $\tau\in\mathrm{T}(\mathrm{M}_{m^{(B)}_{i_{s+2}}}(\mathrm{C}(X^{d^{(B)}_{i_{s+2}}})))$.

Moreover, for each $s = 0, 2, ... $,
$$\phi^{(B, A)}_{i_{s+1}, i_{s+2}} \circ \phi^{(A, B)}_{i_s, i_{s+1}} =\mathrm{diag}\{P_s, R'_s, \Theta'_s \} $$ 
and
$$\phi^{(A)}_{i_{s}, i_{s+2}} = \phi^{(A)}_{i_{s+1}, i_{s+2}} \circ \phi^{(A)}_{i_s, i_{s+1}} = \mathrm{diag}\{P_s, R''_s,  \Theta_s''\}, $$
where $P_s$ is a coordinate projection, and $\Theta'_s$ and $\Theta''_s$ are point evaluations with
$$ \mathrm{rank}(\Theta_s') = \mathrm{rank}(\Theta_s'')\quad\mathrm{and}\quad  \frac{ \mathrm{rank}(R'_s)} {\mathrm{rank}(\Theta'_s)}  = \frac{ \mathrm{rank}(R''_s)} {\mathrm{rank}(\Theta''_s)} < \delta_{i_s},$$
and, furthermore, 
for each $s = 1, 3, ...$,
$$\phi^{(A, B)}_{i_{s+1}, i_{s+2}} \circ \phi^{(B, A)}_{i_s, i_{s+1}} =\mathrm{diag}\{P_s, R'_s, \Theta'_s \} $$ 
and
$$\phi^{(B)}_{i_{s}, i_{s+2}} = \phi^{(B)}_{i_{s+1}, i_{s+2}} \circ \phi^{(A)}_{i_s, i_{s+1}} = \mathrm{diag}\{P_s, R''_s,  \Theta_s''\}, $$
where $P_s$ is a coordinate projection, and $\Theta'_s$ and $\Theta''_s$ are point evaluations with
$$ \mathrm{rank}(\Theta_s') = \mathrm{rank}(\Theta_s'')\quad\mathrm{and}\quad  \frac{ \mathrm{rank}(R'_s)} {\mathrm{rank}(\Theta'_s)}  = \frac{ \mathrm{rank}(R''_s)} {\mathrm{rank}(\Theta''_s)} < \delta_{i_s}.$$

Also note that, by \eqref{bound-tr-A} and \eqref{bound-tr-B}, 
$$ \tau(\phi^{(A)}_{i_{s+1}, i_{s+2}} \circ \phi^{(A)}_{i_s, i_{s+1}}(h)) > \Delta_A(h)/2,\quad h\in\mathcal H^{(A)}_{i_s},\  \tau\in\mathrm{T}(\mathrm{M}_{m_{i_{s+2}}^{(A)}}(\mathrm{C}(X^{d_{i_{s+2}}^{(A)}}))),\ s=0, 2, ... $$
and
$$ \tau(\phi^{(B)}_{i_{s+1}, i_{s+2}} \circ \phi^{(B)}_{i_s, i_{s+1}}(h)) > \Delta_B(h)/2,\quad h\in\mathcal H^{(B)}_{i_s},\  \tau\in\mathrm{T}(\mathrm{M}_{m_{i_{s+2}}^{(B)}}(\mathrm{C}(X^{d_{i_{s+2}}^{(B)}}))),\ s=1, 3, .... $$

By Theorem \ref{unique}, there are unitaries $u_s$, $s=2, 3, ...$, such that
\begin{equation}
 \norm{u_s^*\mathrm{diag}\{P_s, \Theta_s' \}u_s - \mathrm{diag}\{P_s, \Theta_s'' \}} < \frac{\eps_{i_s}}{2} \quad\textrm{on $\mathcal F^{(A)}_{i_s}$}.
\end{equation}
Consider the maps
$$\mathrm{diag}\{R_s', \Theta_s'\} \quad\mathrm{and}\quad \mathrm{diag}\{R_s'', \Theta_s''\}.$$
By Theorem \ref{stable-uniq}, there is a unitary $w_s$ such that
\begin{equation}
\norm{ w_s^*\mathrm{diag}\{R'_s, \Theta_s''\}w_s - \mathrm{diag}\{R''_s, \Theta_s''\} } < \frac{\eps_{i_s}}{2} \quad\textrm{on $\mathcal F^{(A)}_{i_s}$}.
\end{equation}
Setting $v_s = u_sw_s$, and regarding $u_s$ and $w_s$ as unitaries in $\mathrm{M}_{m_{i_s}^{(A)}}(\mathrm{C}(X^{d_{i_s}^{(A)}}))$, we have
$$\norm{v_s^*\mathrm{diag}\{P_s, R_s', \Theta'_s \}v_s- \mathrm{diag}\{P_s, R''_s, \Theta_s'' \}} < \eps_{i_s} \quad\textrm{on $\mathcal F^{(A)}_{i_s}$}.$$ That is,
$$\norm{v_s^*(\phi^{(B, A)}_{i_{s+1}, i_{s+2}} \circ \phi^{(A, B)}_{i_s, i_{s+1}})v_s - \phi^{(A)}_{i_{s+1}, i_{s+2}} \circ \phi^{(A)}_{i_s, i_{s+1}}} < \eps_{i_s} \quad\textrm{on $\mathcal F^{(A)}_{i_s}$}.$$

A similar argument shows that for $s=1, 3, ...$, there are unitaries $v_s$ such that
$$\norm{v_s^*(\phi^{(A, B)}_{i_{s+1}, i_{s+2}} \circ \phi^{(B, A)}_{i_s, i_{s+1}})v_s - \phi^{(B)}_{i_{s+1}, i_{s+2}} \circ \phi^{(B)}_{i_s, i_{s+1}}} < \eps_{i_s} \quad\textrm{on $\mathcal F^{(B)}_{i_s}$}.$$
Therefore, in the diagram
\begin{equation}\label{diag-n}
\xymatrix{
\mathrm{M}_{n_0^{(A)}}(\mathrm{C}(X)) \ar[r]^-{\phi_{1, i_1}^{(A)}} \ar[dr]^-{\phi_{1, i_1}^{(A, B)}} & \mathrm{M}_{m^{(A)}_{i_1}}(\mathrm{C}(X^{d^{(A)}_{i_1}})) \ar[r]^-{\mathrm{ad}(v_2^{(A)})\circ\phi_{i_1, i_2}^{(A)}}  & \mathrm{M}_{m^{(A)}_{i_2}}(\mathrm{C}(X^{d^{(A)}_{i_2}})) \ar[r] \ar[dr]^-{\phi_{i_2, i_3}^{(A, B)}} & \cdots \ar[r] & A \\
\mathrm{M}_{n_0^{(B)}}(\mathrm{C}(X)) \ar[r]_-{\phi_{1, i_1}^{(B)}}  & \mathrm{M}_{m^{(B)}_{i_1}}(\mathrm{C}(X^{d^{(B)}_{i_1}})) \ar[r]_-{\phi_{i_1, i_2}^{(B)}} \ar[ur]^-{\phi_{i_1, i_2}^{(B, A)}}  & \mathrm{M}_{m^{(B)}_{i_2}}(\mathrm{C}(X^{d^{(B)}_{i_2}})) \ar[r]_-{ \mathrm{ad}(v_3^{(B)})\circ \phi_{i_1, i_2}^{(B)}} & \cdots \ar[r] & B,
}
\end{equation}
the $s$th triangle is approximately commutative to within $(\mathcal F_{i_s}^{(A)}, \eps_{i_s})$ or $(\mathcal F_{i_s}^{(B)}, \eps_{i_s})$. By the approximate intertwining argument (\cite{Ell-AT-RR0}), we have $$A \cong B,$$ as desired.
\end{proof}

\begin{rem}
In the case of non-zero radius of comparison, the trace simplex is always the Poulsen simplex (Theorem \ref{trace-VA}).  
\end{rem}

\begin{cor}\label{diff-prod}
Let $X$ be a K-contractible solid metrizable compact space which is finite dimensional. Let 
$$A:=A(X^p, (n^{(A)}_i), (k^{(A)}_i), E^{(A)}) \quad\textrm{and}\quad B:=B(X^q, (n^{(B)}_i), (k^{(B)}_i), F^{(B)})$$
be Villadsen algebras (with coordinate projections of arbitrary multiplicity).
Then $A \cong B$
if, and only if,  
$$\Kzero(A) \cong \Kzero(B),\quad \mathrm{T}(A) \cong \mathrm{T}(B) \quad \textrm{and}\quad \mathrm{rc}(A) = \mathrm{rc}(B).$$
Moreover, if $\mathrm{rc}(A) \neq 0$ (or $\mathrm{rc}(B) \neq 0$), then $\mathrm{T}(A)$ (or $\mathrm{T}(B)$) is redundant in the invariant; that is, $A \cong B$ 
if, and only if,  $$\Kzero(A) \cong \Kzero(B) \quad \textrm{and}\quad \mathrm{rc}(A) = \mathrm{rc}(B).$$
\end{cor}

\begin{proof}
Consider the inductive limit construction
\begin{displaymath}
\xymatrix{
\mathrm{M}_{n_0^{(A)}}(\mathrm{C}(X^p)) \ar[r] & \mathrm{M}_{m_1^{(A)}}(\mathrm{C}(X^{pd^{(A)}_1})) \ar[r]  & \mathrm{M}_{m_2^{(A)}}(\mathrm{C}(X^{pd^{(A)}_2})) \ar[r] & \cdots \ar[r] & A.
}
\end{displaymath}
Tensor it with $\mathrm{M}_{pq}(\Comp)$ and add a new first map, which is induced by coordinate projections with multiplicity one (but no point evaluation yet), to obtain the new construction
\begin{displaymath}
\xymatrix{
\mathrm{M}_{qn_{0}^{(A)}}(\mathrm{C}(X)) \ar[r] & \mathrm{M}_{pq}(\mathrm{M}_{n_0^{(A)}}(\mathrm{C}(X^p))) \ar[r] & \mathrm{M}_{pq}(\mathrm{M}_{m_1^{(A)}}(\mathrm{C}(X^{pd^{(A)}_1}))) \ar[r] & \cdots \ar[r] & \mathrm{M}_{pq}(A).
}
\end{displaymath}
Similarly, also consider $B$, and consider the new inductive limit construction
\begin{displaymath}
\xymatrix{
\mathrm{M}_{pn_{0}^{(B)}}(\mathrm{C}(X)) \ar[r] & \mathrm{M}_{pq}(\mathrm{M}_{n_0^{(B)}}(\mathrm{C}(X^q)) \ar[r] & \mathrm{M}_{pq}(\mathrm{M}_{m_1^{(B)}}(\mathrm{C}(X^{qd^{(B)}_1}))) \ar[r] & \cdots \ar[r] & \mathrm{M}_{pq}(B).
}
\end{displaymath}
Then collapse the first two maps so that there are point evaluations in all connecting maps of the new sequences.

Since $\mathrm{rc}(A) = \mathrm{rc}(B)$, we have $\mathrm{rc}(\mathrm{M}_{pq}(A)) = \mathrm{rc}(\mathrm{M}_{pq}(B))$, and 
by Theorem \ref{n-theorm}, $\mathrm{M}_{pq}(A) \cong \mathrm{M}_{pq}(B)$; denote this algebra by $C$. Note that $[e_A]_0 = [e_B]_0$ in $\Kzero(C)$, where $e_A$ and $e_B$ are the images of $1_A$ and $1_B$ in the upper left corner of the matrix algebras respectively. Since $A$ and $B$ have stable rank one (\cite{EHT-sr1}), $C$ has stable rank one, and hence cancellation of projections. This implies that $e_A$ is Murray-von Neumann equivalent to $e_B$ inside $C$, and therefore $A\cong B$, as asserted.
\end{proof}

\begin{rem}
Let $A$ and $B$ be two Villadsen algebras with seed spaces $X$ and $Y$ respectively. Is $X^\infty \cong Y^\infty$ sufficient for our classification to apply? Or, is it possible that the seed space $X$ is completely irrelevant (when $\mathrm{rc \neq 0}$)?
\end{rem}

\bibliographystyle{plain}

\begin{thebibliography}{10}

\bibitem{Lutley-Alboiu}
M.~Alboiu and J.~Lutley.
\newblock The stable rank of diagonal {ASH} algebras and crossed products by
  minimal homeomorphisms.
\newblock {\em M\"{u}nster J. of Math.}, 15:167--220, 2022.

\bibitem{BW-N}
N.~P. Brown and W.~Winter.
\newblock Quasitraces are traces: a short proof of the finite-nuclear-dimension
  case.
\newblock {\em C. R. Math. Acad. Sci. Soc. R. Can.}, 33(2):44--49, 2011.

\bibitem{CETWW-dim-n}
J.~Castillejos, S.~Evington, A.~Tikuisis, S.~White, and W.~Winter.
\newblock Nuclear dimension of simple {C*}-algebras.
\newblock {\em Invent. Math.}, 224(1):245--290, 2021.

\bibitem{DL-classification}
M.~D{\u{a}}d{\u{a}}rlat and S.~Eilers.
\newblock On the classification of nuclear {C*}-algebras.
\newblock {\em Proc. London Math. Soc. (3)}, 85(1):168--210, 2002.

\bibitem{Dix}
J.~Dixmier.
\newblock On some {C*}-algebras considered by {Glimm}.
\newblock {\em J. Funct. Anal.}, 1:182--203, 1967.

\bibitem{Ell-AT-RR0}
G.~A. Elliott.
\newblock On the classification of {C*}-algebras of real rank zero.
\newblock {\em J. Reine Angew. Math.}, 443:179--219, 1993.

\bibitem{EGLN-DR}
G.~A. Elliott, G.~Gong, H.~Lin, and Z.~Niu.
\newblock On the classification of simple amenable {C*}-algebras with finite
  decomposition rank, {II}.
\newblock 07 2015.

\bibitem{EGLN-ASH}
G.~A. Elliott, G.~Gong, H.~Lin, and Z.~Niu.
\newblock The classification of simple separable unital $\textrm{$\mathcal
  Z$}$-stable locally {ASH} algebras.
\newblock {\em J. Funct. Anal.}, (12):5307--5359, 2017.

\bibitem{EHT-sr1}
G.~A. Elliott, T.~M. Ho, and A.~S. Toms.
\newblock A class of simple {C*}-algebras with stable rank one.
\newblock {\em J. Funct. Anal.}, 256(2):307--322, 2009.

\bibitem{EN-K0-Z}
G.~A. Elliott and Z.~Niu.
\newblock On the classification of simple amenable {C*}-algebras with finite
  decomposition rank.
\newblock In R.~S. Doran and E.~Park, editors, {\em ``Operator Algebras and
  their Applications: A Tribute to Richard V.~Kadison", Contemporary
  Mathematics}, volume 671, pages 117--125. Amer. Math. Soc., 2016.

\bibitem{ENST-ASH}
G.~A. Elliott, Z.~Niu, L.~Santiago, and A.~Tikuisis.
\newblock Decomposition rank of approximately subhomogeneous {C*}-algebras.
\newblock {\em Forum Math.}, 32(4):827--889, 2020.

\bibitem{GK-Dyn}
J.~Giol and D.~Kerr.
\newblock Subshifts and perforation.
\newblock {\em J. Reine Angew. Math.}, 639:107--119, 2010.

\bibitem{Glimm-UHF}
J.~G. Glimm.
\newblock On a certain class of operator algebras.
\newblock {\em Trans. Amer. Math. Soc.}, 95:318--340, 1960.

\bibitem{GL-almost-map}
G.~Gong and H.~Lin.
\newblock Almost multiplicative morphisms and {$K$}-theory.
\newblock {\em Internat. J. Math.}, 11(8):983--1000, 2000.

\bibitem{GLN-TAS-1}
G.~Gong, H.~Lin, and Z.~Niu.
\newblock Classification of finite simple amenable {$\mathcal Z$}-stable
  {C*}-algebras, {I}. {C*}-algebras with generalized tracial rank one.
\newblock {\em C. R. Math. Acad. Sci. Soc. R. Can.}, 42(3):63--450, 2020.

\bibitem{GLN-TAS-2}
G.~Gong, H.~Lin, and Z.~Niu.
\newblock Classification of finite simple amenable {$\mathcal Z$}-stable
  {C*}-algebras, {II}. {C*}-algebras with rational generalized tracial rank
  one.
\newblock {\em C. R. Math. Acad. Sci. Soc. R. Can.}, 42(4):451--539, 2020.

\bibitem{Goodearl-AH}
K.~R. Goodearl.
\newblock Notes on a class of simple {C*}-algebras with real rank zero.
\newblock {\em Publ. Mat.}, 36(2A):637--654 (1993), 1992.

\bibitem{Haagtrace}
U.~Haagerup.
\newblock Quasitraces on exact {$\textrm C^*$}-algebras are traces.
\newblock {\em C. R. Math. Acad. Sci. Soc. R. Can.}, 36(2-3):67--92, 2014.

\bibitem{Hall-Marriage-Lemma}
P.~Hall.
\newblock On representatives of subsets.
\newblock {\em J. London Math. Soc.}, 10(1):26--30, 1935.

\bibitem{Lin-s-uniq}
H.~Lin.
\newblock Stable approximate unitary equivalence of homomorphisms.
\newblock {\em J. Operator Theory}, 47(2):343--378, 2002.

\bibitem{Niu-MD}
Z.~Niu.
\newblock Mean dimension and {AH}-algebras with diagonal maps.
\newblock {\em J. Funct. Anal.}, 266(8):4938--4994, 2014.

\bibitem{Poulsen-Simplex}
E.T. Poulsen.
\newblock A simplex with dense extreme points.
\newblock {\em Ann. Inst. Fourier (Grenoble)}, 11:83--87, 1961.

\bibitem{RS-UCT}
J.~Rosenberg and C.~Schochet.
\newblock The {K}{\"u}nneth theorem and the universal coefficient theorem for
  {K}asparov's generalized {$K$}-functor.
\newblock {\em Duke Math. J.}, 55(2):431--474, 1987.

\bibitem{Sigmund-Poulsen}
K.~Sigmund.
\newblock Generic properties of invariant measures for {A}xiom {${\mathrm A}$}
  diffeomorphisms.
\newblock {\em Invent. Math.}, 11:99--109, 1970.

\bibitem{Sigmund-Poulsen-2}
K.~Sigmund.
\newblock On the space of invariant measures for hyperbolic flows.
\newblock {\em Amer. J. Math.}, 94:31--37, 1972.

\bibitem{TWW-QD}
A.~Tikuisis, S.~White, and W.~Winter.
\newblock Quasidiagonality of nuclear {C*}-algebras.
\newblock {\em Ann. of Math. (2)}, 185(1):229--284, 2017.

\bibitem{RC-Toms}
A.~S. Toms.
\newblock Flat dimension growth for {C*}-algebras.
\newblock {\em J. Funct. Anal.}, 238(2):678--708, 2006.

\bibitem{Toms-Ann}
A.~S. Toms.
\newblock On the classification problem for nuclear {C*}-algebras.
\newblock {\em Ann. of Math. (2)}, 167(3):1029--1044, 2008.

\bibitem{Vill-perf}
J.~Villadsen.
\newblock Simple {C*}-algebras with perforation.
\newblock {\em J. Funct. Anal.}, 154(1):110--116, 1998.

\end{thebibliography}

\end{document}